\documentclass[12pt]{amsart}
\usepackage{amsmath, amssymb, amsthm, amsfonts, enumerate, verbatim}
\usepackage[all]{xy}
\usepackage{tikz}
\tikzstyle{punkt}=[circle, fill=black, minimum size=1mm,inner sep=0pt, draw]
\def\NZQ{\mathbb}               
\def\NN{{\NZQ N}}

\def\ZZ{{\NZQ Z}}
\def\RR{{\NZQ R}}

\def\frk{\frak}               

\def\Phi{{\frk n}}
\def\Phi{{\frk N}}
\def\KK{{\mathbb K}}
\def\opn#1#2{\def#1{\operatorname{#2}}} 
\opn\chara{char}
\opn\length{\ell}
\opn\pd{pd}
\opn\rk{rk}
\opn\projdim{proj\,dim}
\opn\injdim{inj\,dim}
\opn\rank{rank}
\opn\depth{depth} 
\opn\grade{grade}
\opn\height{height}
\opn\embdim{emb\,dim} 
\opn\codim{codim}

\opn\Tr{Tr}
\opn\bigrank{big\,rank}
\opn\superheight{superheight}
\opn\lcm{lcm}
\opn\trdeg{tr\,deg}
\opn\reg{reg}
\opn\lreg{lreg}
\opn\ini{in}
\opn\lpd{lpd}
\opn\size{size}
\opn\bigsize{bigsize}
\opn\cosize{cosize}
\opn\bigcosize{bigcosize}
\opn\sdepth{sdepth}
\opn\sreg{sreg}
\opn\link{link}
\opn\fdepth{fdepth}
\opn\lin{lin}
\opn\ini{in}
\opn\div{div}
\opn\Div{Div}
\opn\cl{cl}
\opn\Cl{Cl}
\opn\Spec{Spec}
\opn\Supp{Supp}
\opn\supp{supp}
\opn\Sing{Sing}
\opn\Ass{Ass}
\opn\Min{Min}
\opn\Mon{Mon}
\opn\dstab{dstab}
\opn\astab{astab}
\opn\Syz{Syz}
\opn\Ann{Ann}
\opn\Rad{Rad}
\opn\Soc{Soc}
\opn\Im{Im}
\opn\Ker{Ker}
\opn\Coker{Coker}
\opn\Am{Am}
\opn\Hom{Hom}
\opn\Tor{Tor}
\opn\Ext{Ext}
\opn\End{End}
\opn\Aut{Aut}
\opn\id{id}

\opn\nat{nat}
\opn\pff{pf}
\opn\Pf{Pf}
\opn\GL{GL}
\opn\SL{SL}
\opn\mod{mod}
\opn\ord{ord}
\opn\Gin{Gin}
\opn\Hilb{Hilb}
\opn\sort{sort}
\opn\initial{init}
\opn\ende{end}
\opn\height{height}
\opn\type{type}
\opn\mdeg{mdeg}
\opn\aff{aff}
\opn\con{conv}
\opn\relint{relint}
\opn\st{st}
\opn\lk{lk}
\opn\cn{cn}
\opn\core{core}
\opn\vol{vol}
\opn\link{link}
\opn\star{star}
\opn\lex{lex}
\opn\sign{sign}
\opn\gr{gr}

\def\pot#1#2{#1[\kern-0.28ex[#2]\kern-0.28ex]}

\opn\dirlim{\underrightarrow{\lim}}
\opn\inivlim{\underleftarrow{\lim}}
\let\iso=\cong

\def\Implies{\ifmmode\Longrightarrow \else
        \unskip${}\Longrightarrow{}$\ignorespaces\fi}
\def\implies{\ifmmode\Rightarrow \else
        \unskip${}\Rightarrow{}$\ignorespaces\fi}
\def\iff{\ifmmode\Longleftrightarrow \else
        \unskip${}\Longleftrightarrow{}$\ignorespaces\fi}

\let\:=\colon
\newtheorem{Theorem}{Theorem}[section]
 \newtheorem{Lemma}[Theorem]{Lemma}
 \newtheorem{Corollary}[Theorem]{Corollary}
 \newtheorem{Proposition}[Theorem]{Proposition}
 \newtheorem{Remark}[Theorem]{Remark}
 
 \newtheorem{Example}[Theorem]{Example}
 
 \newtheorem{Definition}[Theorem]{Definition}
 \newtheorem{Problem}[Theorem]{Problem}

\let\epsilon\varepsilon
\let\kappa=\varkappa
\def\pnt{{\raise0.5mm\hbox{\large\bf.}}}

\begin{document}
\title{Retracts and algebraic properties of cut algebras }
\author {Tim R\"omer and Sara Saeedi Madani}

\address{Tim R\"omer, Universit\"at Osnabr\"uck, Institut f\"ur Mathematik, 49069 Osnabr\"uck, Germany}
\email{troemer@uni-osnabrueck.de}

\address{Sara Saeedi Madani, Universit\"at Osnabr\"uck, Institut f\"ur Mathematik, 49069 Osnabr\"uck, Germany}
\email{sara.saeedimadani@uni-osnabrueck.de}

\begin{abstract}
We study cut algebras which are toric rings associated to graphs. The key idea is to consider suitable retracts to understand algebraic properties and invariants of such algebras like being a complete intersection, having a linear resolution, or the Castelnuovo-Mumford regularity. Throughout the paper, we discuss several examples and pose some problems as well.
\end{abstract}

\thanks{The second author was supported by the German Research Council DFG-GRK~1916.}


\maketitle

\section{Introduction}\label{introduction}

Let $G=(V,E)$ with $V\neq \emptyset$ be a finite simple graph (i.e.~without any loops, directed edges or multiple edges). Given two disjoint subsets $A$ and $B$ of $V$ with $V=A\cup B$, we denote by $A|B$ the (unordered) partition of $V$ by $A$ and $B$. Let $\KK$ be a field. Associated to $G$, we consider two polynomial rings over $\KK$ defined as:
\begin{eqnarray*}
S_G&:=&\KK[q_{A|B}:A\cup B=V,A\cap B=\emptyset],\\
R_G&:=&\KK[s_{ij},t_{ij}:\{i,j\}\in E].
\end{eqnarray*}
Each partition $A|B$ of $V$ defines a subset $\mathrm{Cut}(A|B)$ of the edge set $E$ which is
\[
\mathrm{Cut}(A|B):=\big{\{}\{i,j\}\in E: i\in A,j\in B~\mathrm{or}~i\in B,j\in A\big{\}.}
\]
The set $\mathrm{Cut}(A|B)$ is called a \emph{cut set} of $G$.
Now, consider the following homomorphism of $\KK$-algebras:
\begin{eqnarray*}
\phi_G \colon S_G & \rightarrow & R_G
\\
q_{A|B}& \mapsto &\prod_{\{i,j\}\in \mathrm{Cut}(A|B)}s_{ij}\prod_{\{i,j\}\in E\setminus \mathrm{Cut}(A|B)}t_{ij}.
\end{eqnarray*}
Here the letters ``$s$" and ``$t$" can be thought as abbreviations for ``separated" and ``together", respectively. Observe that the kernel $I_G$ of $\phi_G$ is a graded toric ideal which is called the \emph{cut ideal} of $G$. We call the $\KK$-subalgebra of $R_G$, which is isomorphic to $S_G/I_G$, the \emph{cut algebra} of $G$, and denote it by $\KK[G]$.

Note that the convex hull of the exponent vectors in the $\KK$-algebra generators of $\KK[G]$ is affinely isomorphic to the \emph{cut polytope} ${\mathrm{Cut}}^{\square}(G)$ which was studied a lot in combinatorial optimization; see, e.g., \cite{BM,D,DL1,DL2,N}.
Cut polytopes are known to be full-dimensional polytopes which implies that
the (Krull-)dimension of $\KK[G]$ is equal to $|E|+1$.

Cut ideals were introduced by Sturmfels and Sullivant in \cite{SS}.
In particular, if $G$ is a clique-sum of ``small" subgraphs, then a description of a generating set and Gr\"obner bases of $I_G$
were given in terms of the subgraphs.
Here, by a clique-sum, roughly we mean gluing two graphs in a common complete subgraph, and by small, we mean gluing just in a vertex, an edge or a triangle. In the same paper, the authors presented some applications to algebraic statistics by relating cut ideals to binary graph models, Markov random fields and phylogenetic models on split systems as a generalization of binary Jukes-Cantor models.

One of the most important aspects of the aforementioned paper is presenting some conjectures. More precisely, it was conjectured that $I_G$ is generated in degree at most two if and only if $G$ is a $K_4$-minor-free graph; see \cite[Conjecture~3.5]{SS}. Such graphs are also called series-parallel graphs. Another conjecture posed in \cite{SS} says that $I_G$ is generated in at most degree $4$ if and only if $G$ is $K_5$-minor-free; see \cite[Conjectures~3.6]{SS}. It is also conjectured that $\KK[G]$ is normal and Cohen-Macaulay respectively if and only if $G$ is $K_5$-minor-free; see \cite[Conjectures~3.7]{SS}. The ``only if" part of all conjectures had been verified in \cite{SS}. Recently these conjectures as well as some related questions were studied
by many authors, e.g., in \cite{BC,BHIKS,E,NP,O1,O2,Ol,PS}. Note that Engstr\"om gave an affirmative answer to the first conjecture in \cite{E}.
The other conjectures are still open. The most recent progress on the third conjecture is due to Ohsugi. In \cite{O1} he proved that normality of such algebras is a minor-closed property and he could reduce the relevant part of the third conjecture to $4$-connected plane triangulations. Later on, in \cite{O2} he gave also a characterization of graphs whose cut algebras are normal and Gorenstein.

The authors of the present paper started the study of cut algebras by considering the questions for which pair of graphs $G$ and $H$
the cut algebra of $H$ is an algebra retract of $G$. As was pointed out in \cite[Lemma~3.2~(2)]{SS} this is indeed the case
if $H$ is obtained by an edge contraction from $G$ which then was used, e.g.,
to discuss the highest degree of a minimal generator of the cut ideal
in certain situations. Indeed, in this particular case, it can be shown
that $\mathrm{Cut}^{\square}(H)$ is affinely isomorphic to a face of $\mathrm{Cut}^{\square}(G)$ which yields a particular nice
kind of retracts which are called face retracts in the following (see Corollary~\ref{contraction-retract}).

Note that in \cite[Lemma~3.2~(1)]{SS} it was also stated that
for an induced subgraph $H$ of $G$ we have that $\mathrm{Cut}^{\square}(H)$ is affinely isomorphic to a face of $\mathrm{Cut}^{\square}(G)$. Unfortunately, this is not true and we provide a counterexample to this in Example~\ref{Ohsugi}. We would like to mention that in spite of the existence of this example,
no further problems arise in the main results or other statements in \cite{SS}
besides the one (see \cite[Corollary~3.3~(1)]{SS}) that now it remains as an open problem whether the highest degree in a minimal system of generators of $I_H$ does not exceed the one of $I_G$ if $H$ is an induced subgraph of $G$.
We do not know any counterexample to this question, nor a proof of it.
But we give certain (sufficient) conditions for an induced subgraph  $H$ of a graph $G$
to induce a retract on the level of cut algebras (see Theorem~\ref{golden1}).
We give several applications of our results, concerning interesting algebraic properties and invariants of cut algebras and
cut ideals. Beside this we also discuss several examples as well as posing some
problems.

The paper is organized as follows. In Section~\ref{Notations}, we present some definitions, well-known facts and statements, and notation. This section is divided into four subsections: graphs, cut sets and algebras, polytopal algebras, and algebra retracts.

In Section~\ref{properties and examples}, we discuss some basic properties of cut algebras such as their different gradings from which we benefit in the next sections. In this section we also determine exactly  when the cut ideal is zero and when it has linear forms as generators.

In Section~\ref{contraction}, we give a new proof for the fact that the cut polytope of a graph obtained by an edge contraction of a graph $G$ is affinely isomorphic to a face of the cut polytope of $G$ (see \cite[Lemma~3.2~(2)]{SS}), which implies the existence of a face retract on the level of cut algebras.

In Section~\ref{induced}, first we discuss the aforementioned counterexample. Then we introduce a certain type of minors of a graph which we call ``neighborhood-minors", and we show that the cut  algebras of such minors of a graph $G$ are algebra retracts of the cut algebra of $G$. As special cases of such minors, we mention, e.g., clique-sums and vertex duplications. We also discuss several well-known classes of graphs whose all/some of induced subgraphs provide such an algebra retract, like chordal graphs, complete $t$-partite graphs, Ferrers graphs and ring graphs.

In Section~\ref{application}, we give some applications of the results from the former sections. First, we introduce the notion of a ``combinatorial retract" of a graph which in particular is also a minor, and is constructed via the edge contractions and neighborhood-minors repeatedly. Therefore, it yields algebra retracts of the cut algebra of the original graph. As a consequence we verify a weaker version of \cite[Conjecture~3.1]{SS}.
In this section, we determine when the cut ideal of a graph $G$
is generated in a single degree.
We also characterize all graphs whose cut algebras are complete intersection. Furthermore, we classify all connected graphs whose cut ideals have linear resolution, and in particular, those ones which have $2$-linear resolution, namely those having regularity $2$. In this section, we also discuss some examples for which we can get nontrivial lower bounds for their regularity by applying our results. We end this section by giving a necessary condition for the cut ideals which satisfy property $N_1$, and pose a problem based on our computational experiments on the converse of the latter statement.

In Section~\ref{Examples}, some examples on some algebraic properties of cut ideals of certain graphs are presented which are essential in the literature like complete graphs, as well as those ones which are used and play important roles throughout this paper.

The authors thank Hidefumi Ohsugi and Bernd Sturmfels for valuable comments and suggestions for this paper.

\section{Preliminaries and Notations}\label{Notations}

In this section we recall some definitions and facts which are used in the rest of the paper. We also fix some notations.

\subsection{Graphs}\label{graphs}

Throughout the paper, all graphs are assumed to be simple with a non-empty set of vertices. If the set of edges of a graph is empty, then it means that the graph is a trivial graph which consists of isolated vertices.

For a graph $G$, we sometimes write $V(G)$ and $E(G)$ to denote the set of vertices and edges, respectively. In the following, we fix the  notations to denote some certain types of graphs which are used throughout this paper. Here, $G=(V,E)$, and $n:=|V|$.

\begin{itemize}
\item $K_n$: Complete graph.
\item $C_n$: $n$-Cycle (cycle of length $n$).
\item $P_n$: Path of length $n-1$.
\item $K_{1,n-1}$: Star graph, which is the complete bipartite graph with partitions of cardinalities $1$ and $n-1$.
\item $G\setminus e$: the graph with the vertex set $V$ and the edge set $E\setminus e$ where $e\in E$.
\item $G=G_1\sqcup \cdots \sqcup G_r$: disjoint union of graphs $G_1,\ldots, G_r$.
\item $rG$: $\bigsqcup_{i=1}^r G$, namely disjoint union of $r$ copies of $G$.
\end{itemize}

Note that, because of symmetry in the complete graph $K_n$, we simply mean by $K_n\setminus e$ the graph which is obtained by deletion of an arbitrary edge of $K_n$.

Let $G=(V,E)$ be a graph and $v\in V$. Then
\[
N_G(v):=\{w\in V: w~\mathrm{is~a~neighbor~of}~v~\mathrm{in}~G\},
\]
where a vertex $w\in V$ is called a \emph{neighbor} of $v$ in $G$, if it is adjacent to $v$. Moreover,
\[
N_G[v]:=N_G(v)\cup \{v\} ~~ \mathrm{and}~~ N_G(T):=\cup_{v\in T} N_{G}(v),
\]
for any non-empty subset $T$ of $V$.
Then, the \emph{degree} of a vertex $v$ of $G$ is defined to be $\deg_{G}(v):=|N_{G}(v)|$.
Let $W$ be a nonempty subset of $V$. Recall that an \emph{induced subgraph} of $G$ on $W$, denoted by $G_W$, is the graph on the vertex set $W$ whose edges are exactly those edges of $G$ whose vertices are in $W$. If the induced subgraph $G_W$ is a complete graph on $W$, then it is called a \emph{clique} of $G$.

\subsection{Cut sets and algebras}\label{cut}

Now, keeping in mind the notations introduced in Section~\ref{introduction}, we fix some further notations. Let $G=(V,E)$ be a graph on $n$ vertices, with a partition $V=A\cup B$ of its vertex set. Then, clearly, $B=A^c:=V\setminus A$ is the complement of $A$ which is just determined by $A$. So
we set
\[
\mathrm{Cut}(A)=\mathrm{Cut}(A^c):=\mathrm{Cut}(A|A^c),
\]
and also for the variables of the polynomial ring $S_G$, we put
\[
q_A:=q_{A|A^c}.
\]
In particular, $q_A$ and $q_{A^c}$ are the same. Moreover, $\phi_G(q_A)$ is by the definition a monomial in $R_G$, which we denote by
\[
u_A:=\phi_G(q_A).
\]
In particular, $u_A=u_{A^c}$.
Then we have $\KK[G]=\KK[u_A: A\subseteq V]$.

Note that there are $2^{n-1}$ distinct partitions for $V$ which bijectively correspond to the variables of the polynomial ring $S_G$. But, the generators of the $\KK$-algebra $\KK[G]$ do not always bijectively correspond to them unless $G$ is connected (see Proposition~\ref{linear forms}).
In addition, in some explicit cases, if $A=\{i_1,\ldots,i_k\}$, then we write  $q_{i_1\ldots i_k}$ and $u_{i_1\ldots i_k}$ instead of $q_{\{i_1,\ldots,i_k\}}$ and $u_{\{i_1,\ldots,i_k\}}$, respectively.

Finally, we would like to point out that the cut ideal $I_G$ is a prime ideal generated by some \emph{pure binomials} in $S_G$, since it is a toric ideal. Here, by a pure binomial in a polynomial ring we mean a binomial of the form $u-v$ where $u$ and $v$ are monomials.

\subsection{Polytopal algebras}\label{polytope}

Let $P\subset \RR^d$ be a polytope, that is $P=\mathrm{conv}(v_1,\ldots,v_r)$,  the convex hull of $v_1,\ldots,v_r\in \RR^d$. The \emph{dimension} of $P$ is the dimension of the \emph{affine hull} $\mathrm{aff}(P)$, which is the smallest affine subspace of $\RR^d$ containing $P$. Moreover, a \emph{morphism} of polytopes $P$ and $Q$ is a map $\varphi\colon P\rightarrow Q$ that can be extended to an affine map $\tilde{\varphi}\colon \mathrm{aff}(P)\rightarrow \mathrm{aff}(Q)$. In particular, the polytopes $P$ and $Q$ are said to be \emph{affinely isomorphic} if the morphism $\varphi$ is bijective.

We also recall that a \emph{polyhedron} is the set of solutions of a linear system of inequalities, and a polytope is equivalently a \emph{bounded polyhedron}.

A subset $F$ of a polytope $P$ is a \emph{face} of $P$ if $F$ is the intersection of $P$ with a hyperplane $H$ (a \emph{supporting hyperplane}), such that $P$ is entirely contained in one of the two half-spaces of $\RR^d$ given by $H$. Observe that each face of a polytope is a polytope itself. Faces of dimension $0$ and $\dim P$ are called \emph{vertices} and \emph{facets}, respectively. Moreover, the empty set and $P$ itself are trivial faces.

A polytope $P$ is called a \emph{lattice polytope} if its vertices are lattice points of the integral lattice $\ZZ^d\subset \RR^d$. The set of lattice points in $P$ are denoted by $L_P$.

Now, let $P\subset \RR^d$ be a lattice polytope. Associated with $P$ is a standard graded $\KK$-algebra $\KK[P]$ (up to graded isomorphism) called \emph{polytopal algebra} whose generators correspond bijectively to the lattice points in $P$. More precisely,
\[
\KK[P]=\KK[{\bf{y^a}}z:{\bf{a}}\in P\cap \ZZ^d]
\]
is a $\KK$-subalgebra of $\KK[y_1,\ldots,y_d,z]$ where $${\bf{y^a}}:=y^{a_1}_1\cdots y^{a_d}_d
\text{ for } {\bf a}=(a_1,\ldots ,a_d).
$$
 Here, standard graded means graded and generated in degree $1$.

\subsection{Algebra retracts}\label{retract}

First we recall the well-known definition of an algebra retract of a graded algebra. If not stated otherwise, by ``graded" we mean the standard $\ZZ$-graded (see Section~\ref{properties and examples}).

\begin{Definition}\label{retract-def}
{\em Let $A$ and $B$ be graded $\KK$-algebras and let
$\iota\colon A\rightarrow B$ be an injective homogeneous $\KK$-algebra homomorphism. Then $A$ is called an \emph{algebra retract} of $B$, if there exists a homogeneous (surjective) homomorphism of $\KK$-algebras
$\pi\colon B\rightarrow A$ such that $\pi \circ \iota=\mathrm{id}_{A}$.
}
\end{Definition}

It follows clearly from the definition that if $A$, $B$ and $C$ are graded $\KK$-algebras where $A$ is an algebra retract of $B$, and $B$ is an algebra of $C$, then $A$ is an algebra retract of $C$. Note that we do not insist that the homogeneous homomorphisms are of degree $0$. Moreover, in this paper, we do not consider more general kinds of algebra retracts, i.e.~not graded ones.

Several algebraic properties, such as being regular, a complete intersection and Koszul, are known to be preserved by algebra retracts (see for example \cite{EN} and \cite{OHH}). Moreover, in \cite[Corollary~2.5]{OHH} it is shown that the graded Betti numbers do not increase by retraction. The precise statement is the following.

\begin{Proposition}\label{Betti}
{\em (}\cite[Corollary~2.5]{OHH}{\em )}
Let $R=A/I$ and $S=B/J$ be graded $\KK$-algebras where $A$ and $B$ are polynomial rings over a field $\KK$. Suppose that $R$ is an algebra retract of $S$, and $I$ and $J$ are graded ideals containing no linear forms. Then:
\begin{enumerate}\label{betti-retract}
\item [{\em (a)}] $\beta_{i,j}^A(I)\leq \beta_{i,j}^B(J)$ for all $i,j$.
\item [{\em (b)}] $\reg_A(I)\leq \reg_B(J)$.
\item [{\em (c)}] $\projdim_A(I)\leq \projdim_B(J)$.
\end{enumerate}
\end{Proposition}

Here, we recall that for a polynomial ring $R=\KK[x_1,\ldots,x_n]$ and a finitely generated $R$-module $M$, in general, the \emph{graded Betti numbers} are defined as:
\[
\beta_{i,j}^R(M):=\dim_{\KK} \Tor_i^R(\KK,M)_j.
\]
The \emph{Castelnuovo-Mumford regularity} and the \emph{projective dimension} of $M$ are defined as:
\[
\reg_R(M):=\max\{j-i:\beta_{i,j}^R(M)\neq 0\},
\]
\[
\projdim_R(M):=\max\{i:\beta_{i,j}^R(M)\neq 0~\mathrm{for~some~}~j\}.
\]

Note that the properties of being Gorenstein and Cohen-Macaulay are not preserved necessarily by algebra retracts, (see, e.g., \cite[Example~3.9]{EN}).

As an example of well-known algebra retracts one can mention face retracts. Let $P$ be a polytope. Given a face $F\subseteq P$, it follows from \cite[Corollary~4.34]{BG} that there is a natural (standard graded) algebra retract $\KK[F]$ of $\KK[P]$ with the surjective homomorphism
\begin{equation}\label{K[F]}
\pi\colon \KK[P]\rightarrow \KK[F]
\end{equation}
with $\pi({\bf{ y^a}}z)=0$ if ${\bf a}\in L_P\setminus F$. Such an algebra retract is called a \emph{face retract}.

At the end of this section, we refer the reader to, e.g., \cite{BG,CLS,S} for more details about toric algebras and related topics.

\section{Some basic properties of Cut Algebras}\label{properties and examples}

In this section, we provide some fundamental properties of cut ideals and algebras.

First we focus on different types of gradings that one can associate to cut algebras. Let $G=(V,E)$ be a graph with $E=\{e_1,\ldots,e_{m}\}$ where $m\geq 1$, and let $A\subset V$. We are interested in the following gradings of the cut algebra of $G$ which imply that the $\KK$-algebra homomorphism $\phi_G$ is homogeneous, and hence $I_G$ is a graded ideal of $S_G$ with respect to the desired graded rings.

\begin{itemize}
\item \emph{The standard $\ZZ$-grading}: We put $\deg(q_A)=\deg(u_A):=1$. Then $\KK[G]$ is a standard graded $\KK$-algebra.

\item  \emph{The $\ZZ^{2|E|}$-(multi-)grading}: We set $\mdeg(q_A)=\mdeg(u_A):=\varepsilon_A$, where $\varepsilon_A$ is a vector in $\ZZ^{2|E|}$ such that for $i=1,\ldots, m$,
\begin{displaymath}
{(\varepsilon_A)}_i= \left \{\begin {array}{ll} 1 & \mathrm{if}~~~
e_i\in \mathrm{Cut(A)},\\
0 & \mathrm{otherwise},
\end{array}\right.
\end{displaymath}
and
\[
{(\varepsilon_A)}_{m+i}=1-{(\varepsilon_A)}_i.
\]
Here, $\mdeg$ stands for the \emph{multidegree}.

\item  \emph{The $(s,t)$-bi-grading}: We define the \emph{$s$-degree} part as \[
\deg_s(q_A)=\deg_s(u_A):=|\mathrm{Cut(A)}|,
\]
the \emph{$t$-degree} part as
\[
\deg_t(q_A)=\deg_t(u_A):=2|E|-|\mathrm{Cut(A)}|,
\]
and
\[
\deg_{s,t}(q_A):=\bigl(\deg_s(q_A),\ \deg_t(q_A)\bigr).
\]
\end{itemize}

If not stated otherwise, throughout the paper we use the standard $\ZZ$-grading.
Next, we recall dimension and height formulas from the literature which will be used in the sequel. Indeed, for a graph $G=(V,E)$ with $|V|=n$, it is known that
\begin{equation}\label{dim}
\dim S_G/I_G=|E|+1,
\end{equation}
(see, e.g., \cite[Proposition~4.22]{BG} together with the fact that $\mathrm{Cut}^{\square}(G)$ has dimension~$|E|$). It follows that
\begin{equation}\label{height}
\height I_G=2^{n-1}-|E|-1
\end{equation}
and thus
\begin{equation}\label{projdim}
\projdim_{S_G} S_G/I_G\geq 2^{n-1}-|E|-1,
\end{equation}
where the equality holds if and only if $\mathbb{K}[G]$ is Cohen-Macaulay.

One of the main points which has been considered in \cite{SS} is the highest degree of a generator in a minimal system of generators of the cut ideal of a graph. Moreover, there is a characterization of the graphs whose cut ideals are generated in degree~$\leq 2$.

Here, we provide the two following facts which determine when either $I_G=\langle 0 \rangle$ or there are linear forms in the generating set of a cut ideal. The proofs are straightforward, but not discussed in the literature.

\begin{Proposition}\label{regularity0}
Let $G=(V,E)$ be a graph with $|E|\geq 1$. Then the following statements are equivalent:
\begin{enumerate}
\item[{\em (a)}] $I_G=\langle 0 \rangle$, i.e. $\mathbb{K}[G]$ is regular;
\item[{\em (b)}] $G=K_2$ or $G=K_3$.
\end{enumerate}
\end{Proposition}

\begin{proof}
Suppose that $|V|=n$. To verify the statement, it is enough to show that $\height I_G=0$ if and only if $G=K_2$ or $K_3$.

If $n=2$, then $G=K_2$ for which $\height I_G=0$ by~(\ref{height}).

If $n=3$, then $G$ can be $K_3$ or $P_3$ or $K_2\sqcup K_1$. In this case, by~(\ref{height}), $\height I_G=0$ if and only if $G=K_3$.

For $n\geq 4$, we show that $2^{n-1} > \binom{n}{2}+1$, which implies that $2^{n-1}>|E|+1$, since $\binom{n}{2}$ is the number of edges of the complete graph $K_n$. Then it follows that $\height I_G\neq 0$, and we are done. For $n=4$, the desired inequality clearly holds. Now, we assume $n\geq 5$. Note that $\binom{n}{2}$ is the number of subsets of $V$ of cardinality~$2$, and $2^{n-1}$ is the number of all partitions of $V$. Since $n\geq 5$, there is no partition of $V$ into two subsets of cardinalities~$2$. So, each subset of cardinality~$2$ of $V$ determines exactly one partition for $V$. Besides these partitions there are several other partitions for $V$, and hence the desired inequality holds.
\end{proof}

After knowing the simple structure of the cut algebras associated to $K_2$ and $K_3$ in Proposition~\ref{regularity0}, it worths to understand better the cut algebras of other complete graphs. The cases of the next two small complete graphs, namely $K_4$ and $K_5$, we investigate some of their properties in Section~\ref{Examples} (see Examples~\ref{K_4}~and~\ref{K_5}, and Problem~\ref{problemK_n}). These cut algebras are of special interest, because there are interesting known results and conjectures on cut ideals in which these two complete graphs play prominent roles, (see, e.g., \cite{E} and \cite{SS}).

In the following, we observe that the only case where linear forms belong to a generating set of a cut ideal is when the graph is disconnected.

\begin{Proposition}\label{linear forms}
Let $G=(V,E)$ be a graph with $|E|\geq 1$. Then the following statements are equivalent:
\begin{enumerate}
\item[{\em (a)}] ${(I_G)}_1\neq 0$;
\item[{\em (b)}] $G$ is disconnected.
\end{enumerate}
\end{Proposition}

\begin{proof}
(b) $\Rightarrow$ (a) follows from \cite[Proposition~5.2]{NP} (see also Proposition~\ref{disconnected}).

Now, we prove (a) $\implies$ (b). Suppose there is a generator of degree $1$ in $I_G$, namely $q_{A}-q_{C}$ for two different partitions given by subsets $A$ and $C$ of $V$. Therefore, ${\phi}_G(q_{A}-q_{C})=0$, and hence $u_{A}=u_{C}$. Using the bi-grading of $\KK[G]$, it follows that $\mathrm{Cut}(A)=\mathrm{Cut}(C)$. If $G$ is connected, then one can observe that any cut set of edges is given by a unique partition of $V$. Therefore, since the partitions given by $A$ and $C$ are different, it follows that $G$ is a disconnected graph.
\end{proof}

We end this section by a discussion of the projective dimension of cut algebras. By~(\ref{projdim}), it is reasonable to know the graphs whose cut algebras have small projective dimension. In the following, we discuss this question. Indeed, in Proposition~\ref{regularity0} the graphs whose cut algebra has projective dimension $0$ have been determined.

\begin{Proposition}\label{projective dimension}
Let $G=(V,E)$ be a graph with $|V|=n$ and $|E|\geq 1$, and let $p:=\projdim_{S_G} S_G/I_G$. Then we have:
\begin{enumerate}
\item[{\em (a)}] If $n\leq 5$ and $G\neq K_5$, then
$p=2^{n-1}-|E|-1$, and hence $p\leq 14$.
\item[{\em (b)}] If $G=K_5$, then $p=15$.
\item[{\em (c)}] If $n=6$, then $p\geq 16$.
\item[{\em (d)}] If $n\geq 7$, then $p\geq 42$.
\end{enumerate}
\end{Proposition}
\begin{proof}
(a) Suppose $n\leq 5$ and $G\neq K_5$. Then, by \cite[Example~3.7]{O1}, we have that $\mathbb{K}[G]$ is normal. This, together with a theorem of Hochster in \cite{Ho} (see also, e.g., \cite[Theorem~6.10]{BG}) implies that $\mathbb{K}[G]$ is Cohen-Macaulay, and hence in (\ref{projdim}) the equality occurs. So $p=2^{n-1}-|E|-1$. This obviously implies that $p\leq 14$.

(b) This part follows from Example~\ref{K_5}.

(c) Suppose that $n=6$. Then by (\ref{projdim}) we get $p\geq 2^{5}-\binom{6}{2}-1=16$, since $|E|\leq \binom{6}{2}$.

(d) Suppose $n\geq 7$. Then, similar to the previous part, by (\ref{projdim}) we get $p\geq 2^{n-1}-\binom{n}{2}-1$, since $|E|\leq \binom{n}{2}$. Note that, since $n\geq 7$, each subset of cardinality~$2$ of $V$ determines exactly one partition for $V$. Therefore, $2^{n-1}-\binom{n}{2}$ is exactly the number of those partitions
$V=A\cup B$ where $|A|,|B|\neq 2$. If $|A|=0$ or $1$ or $3$, then $|B|\neq |A|$ and $|B|\neq 2$, since $n\geq 7$. Thus, there are at least $1+7+\binom{7}{3}=43$ partitions with the latter property, and hence $p\geq 2^{n-1}-\binom{n}{2}-1\geq 42$.
\end{proof}

\begin{Remark}
{\em  We would like to remark that by Proposition~\ref{projective dimension}, we see that for a given $k$ with $k=0,\ldots,15$, there is at least one graph whose cut algebra is of projective dimension~$k$. Those which have no isolated vertices are presented in Table~\ref{table:1}. Furthermore, it follows from Proposition~\ref{projective dimension} together with (\ref{projdim}) that the only graph whose cut algebra could have projective dimension~$16$ is $K_6$. But, based on computations in \cite[Table~1]{SS}, the cut algebra $\mathbb{K}[K_6]$ is not Cohen-Macaulay. So that projective dimension is bigger than $16$ in this case, which implies that projective dimension $16$ can not occur among the cut algebras. }
\end{Remark}

\section{Face retracts of cut polytopes}\label{contraction}

In this section we recall some properties of cut polytopes of graphs and afterwards discuss certain face retracts of cut algebras arising from those polytopes.

First we recall the definition of a cut polytope. For more information about cut polytopes, we refer the reader, e.g., to \cite{DL2}.

\begin{Definition}\label{cut polytope}
{\em Let $G=(V,E)$ be a graph. Then the \emph{cut polytope} ${\mathrm{Cut}}^{\square}(G)$ of $G$ is the convex hull of the \emph{cut vectors} $\delta_{A}\in \mathbb{R}^{|E|}$ of $G$, which are defined as
\begin{displaymath}
\delta_{A}(\{i,j\})= \left \{\begin {array}{ll}
1&\mathrm{if}~~~|A\cap \{i,j\}|=1,\\
0&\mathrm{otherwise},
\end{array}\right.
\end{displaymath}
for any $A\subseteq V$ and $\{i,j\}\in E$. }
\end{Definition}

Note that there is a natural bijection between the cut vectors and the cut sets of $G$. Moreover, observe that the set of vertices of the polytope  ${\mathrm{Cut}}^{\square}(G)$ coincides with the set of all cut vectors of $G$, since cut polytopes are $\{0,1\}$-polytopes. So, if $G$ is a connected graph, then by Proposition~\ref{linear forms} we get the well-known fact that ${\mathrm{Cut}}^{\square}(G)$ has exactly $2^{n-1}$ vertices.

In general, we have that
\[
|E|+1\leq \mathrm{number~of~vertices~of}~{\mathrm{Cut}}^{\square}(G) \leq 2^{n-1},
\]
where the first inequality follows from the fact that ${\mathrm{Cut}}^{\square}(G)$ is full-dimensional.

Next we observe how the polytopal algebra associated to ${\mathrm{Cut}}^{\square}(G)$ is related to the cut algebra of the underlying graph.

\begin{Lemma}\label{cut polytopal algebra}
Let $G$ be a graph. Then there exists a natural isomorphism
\[
\KK[G]\iso \KK[{\mathrm{Cut}}^{\square}(G)]
\]
as standard graded $\KK$-algebras.
\end{Lemma}

\begin{proof}
We define:
\begin{eqnarray*}
\varphi\colon \KK[G]&\rightarrow& \KK[{\mathrm{Cut}}^{\square}(G)]
\\
\\
u_A=\prod_{\{i,j\}\in \mathrm{Cut}(A)}s_{ij}\prod_{\{i,j\}\in E\setminus \mathrm{Cut}(A)}t_{ij} &\mapsto& z\prod_{\{i,j\}\in \mathrm{Cut}(A)}s_{ij},
\end{eqnarray*}
for any $A\subseteq V(G)$.

Then, it is easily seen that $\varphi$ is well-defined and a homogeneous surjective homomorphism. Injectivity of $\varphi$ follows from the fact that for each $A\subseteq V(G)$, variables $t_{ij}$ appearing in $u_A$ as factors are uniquely determined by $A$.
\end{proof}

An operation from graph theory which has played an important role in the study of cut polytopes as well as cut algebras is taking \emph{minors}. To recall the definition of a minor of a graph, we need to recall two other operations, namely ``edge deletion" and ``edge contraction".

First recall that a graph $H$ is said to be obtained by an \emph{edge deletion} from a graph $G$, if $H=G\setminus e$ for some $e\in E(G)$.

Next we recall the definition of the edge contraction operation. Let $G=(V,E)$ be a graph, and let $e=\{u,v\}\in E$. A graph $G'=(V',E')$ is said to be obtained by an \emph{edge contraction} from $G$ if
\[
V':=(V\setminus \{u,v\})\cup \{w\},
\]
and
\[
E':=\big{\{}e\in E:e\cap N_G(\{u,v\})=\emptyset\big{\}}\cup \big{\{}\{w,z\}:z\in V'\cap N_G(\{u,v\})\big{\}}.
\]
This means that the edge contraction operation (relative to an edge $e$) is constructed as follows. The edge $e$ is removed from $G$ and its two vertices, $u$ and $v$, are merged into a new vertex $w$, where the edges incident to $w$ in the graph $G'$ each corresponds to an edge incident to either $u$ or $v$.

Note that in some context, it is also allowed to get multiple edges after the contraction of an edge, but here we always consider the simple graph (without any multiple edges) obtained by this operation.

For example, by contracting any edge of the cycle $C_n$, we obtain the cycle $C_{n-1}$. Also, by contracting any edge of the complete graph $K_n$, one obtains the complete graph $K_{n-1}$.

Finally, we recall the definition of a minor of a graph. Let $G$ be a graph. Then a \emph{minor} of $G$ is a graph obtained from $G$ by applying a sequence of the operations edge deletion and edge contraction.

Moreover, given another graph $H$, then $G$ is called \emph{$H$-minor-free} if it does not have any minor isomorphic to $H$.

\begin{Remark}\label{contraction K_n}
{\em Observe that if a graph has a complete graph as a minor, then this minor can be obtained just by a sequence of edge contractions (see also \cite[page~699]{SS}). }
\end{Remark}

\begin{Remark}
{\em  It was already mentioned in \cite[page~699]{SS} that ${\mathrm{Cut}}^{\square}(G\setminus e)$, for some $e\in E(G)$, is not corresponding to a face of ${\mathrm{Cut}}^{\square}(G)$, so that one does not expect to get a face retract by edge deletions. But, even more generally, $\KK[G\setminus e]$ is not necessarily an algebra retract of $\KK[G]$. For example, let $G=C_4$, and let $e$ be any of its edges. Then $G\setminus e =P_4$. But, Proposition~\ref{projective dimension}~(a) shows that
\[
\projdim_{S_{C_4}}(I_{C_4})=3<\projdim_{S_{P_4}}(I_{P_4})=4,
\]
which implies by Proposition~\ref{betti-retract} that $\KK[P_4]$ is not an algebra retract of $\KK[C_4]$. }
\end{Remark}

The above remark also implies that graph minors do not provide algebra retracts in general, but still in some cases one gets retracts. One of these cases is when the minors are just obtained by edge contractions repeatedly. Indeed, the statement in \cite[Lemma~3.2~(2)]{SS} says that if $G'$ is a graph obtained by an edge contraction from the graph $G$, then ${\mathrm{Cut}}^{\square}(G')$ is a ``face" of ${\mathrm{Cut}}^{\square}(G)$. More precisely, the next proposition holds. Since there is a problem with some parts of the proof of \cite[Lemma~3.2]{SS} (see Example~\ref{Ohsugi}) and for the convenience of the reader we provide here a precise proof of the mentioned fact. In the next section, we deal with another case which gives an algebra retract.

\begin{Proposition}\label{face-contraction}
Let $G$ be a graph, and assume that a graph $G'$ is obtained from $G$ by an edge contraction. Then ${\mathrm{Cut}}^{\square}(G')$ is affinely isomorphic to a face of ${\mathrm{Cut}}^{\square}(G)$.
\end{Proposition}

\begin{proof}
Let $G=(V,E)$ with $|V|=n$, and $e=\{u,v\}\in E$. Assume that $G'=(V',E')$ is the graph obtained by contracting the edge $e$, and merging the vertices $u,v$ into $w$. Also suppose that $N_G(u)\cap N_G(v)=\{v_1,\ldots,v_p\}$; here $p$ can be also $0$, which means that $u$ and $v$ have no common neighbors in $G$.

Since the coordinates $x_i$ of a vector $x$ in $\RR^{|E'|}$ and $\RR^{|E|}$ are corresponding to the edges of $G'$ and $G$, respectively, we can without loss of generality assume that the coordinates are organized as follows:

For $\RR^{|E'|}$, we associate
\begin{itemize}
\item the coordinates $x_1,\ldots, x_p$ to the edges $\{w,v_i\}$ for $i=1,\ldots,p$,
\item the coordinates $x_{p+1},\ldots, x_{|E'|}$ to the edges which do not contain $w$ (arbitrarily ordered).
\end{itemize}

For $\RR^{|E|}$, we associate
\begin{itemize}
\item the coordinates $x_1,\ldots, x_p$ to the edges $\{u,v_i\}$ for $i=1,\ldots,p$,
\item the coordinates $x_{p+1},\ldots, x_{|E'|}$ to the edges which do not contain $w$ (ordered in the same way as the``corresponding" coordinates in $\RR^{|E'|}$),
\item the coordinate $x_{|E'|+1}$ to the edge $e$,
\item the coordinates $x_{|E'|+2},\ldots, x_{|E|}$ to the edges $\{v,v_i\}$ with for $i=1,\ldots,p$.
\end{itemize}

Now, let $H$ be the hyperplane in $\RR^{|E|}$ given by $x_{e}=0$, and let
$F:={\mathrm{Cut}}^{\square}(G)\cap H$. Then it follows that $F$ is a face of ${\mathrm{Cut}}^{\square}(G)$, as ${\mathrm{Cut}}^{\square}(G)$ is contained in one of the half-spaces defined by $H$.

We show that ${\mathrm{Cut}}^{\square}(G')$ is affinely isomorphic to $F$. For simplicity, let $Q:={\mathrm{Cut}}^{\square}(G')\subseteq \RR^{|E'|}$.

Suppose that $\alpha_0,\alpha_1,\ldots,\alpha_r\in \RR^{|E'|}$ are the vertices of $Q$. Here $|E'|+1\leq r\leq 2^{n-2}$ where $|V'|=n-1$. In addition, assume that $\alpha_0$ is the zero vector (corresponding to $\mathrm{Cut}({\emptyset}$)), and $\alpha_1,\ldots,\alpha_{|E'|}$ provide a basis for $\mathrm{aff}(Q)$ as a vector space. Here, $\mathrm{aff}(Q)=\mathrm{span}(Q)$, since zero is an element of $\mathrm{aff}(Q)$.

For simplicity, for any
\[
\beta=(\beta_1,\ldots,\beta_p,\beta_{p+1},\ldots,\beta_{|E'|})\in \RR^{|E'|},
\]
we write
\[
\beta^{(p)}:=(\beta_{1},\ldots,\beta_{p}).
\]
In particular, for each
\[
\alpha_i=(\alpha_{i1},\ldots,\alpha_{ip},\alpha_{ip+1},\ldots,\alpha_{i|E'|}),
\]
where $i=0,\ldots,r$, we set
\[
\alpha_i=(\alpha^{(p)}_i,\alpha_{ip+1},\ldots,\alpha_{i|E'|}).
\]
Now, we define the following linear map:
\begin{eqnarray*}
\tilde{\varphi}\colon \mathrm{aff}(Q)\iso \RR^{|E'|}&\rightarrow & \mathrm{aff}(F)\subseteq \RR^{|E|}\\
\alpha_i &\mapsto&(\alpha_i,0,\alpha^{(p)}_i), \quad  \mathrm{for}~i=1,\ldots,|E'|.
\end{eqnarray*}

Assume that $\alpha_i=\sum_{j=1}^{|E'|}c_j\alpha_j$, for $i=|E'|+1,\ldots,r$, where $c_j\in \RR$ for all $j$. Then it follows that
\[
\tilde{\varphi}(\alpha_i)=\sum_{j=1}^{|E'|}c_j\tilde{\varphi}(\alpha_j)=
\sum_{j=1}^{|E'|}c_j(\alpha_j,0,\alpha^{(p)}_j)=(\sum_{j=1}^{|E'|}c_j\alpha_j,0,\sum_{j=1}^{|E'|}c_j\alpha^{(p)}_j),
\]
which implies that the map
\begin{eqnarray*}
\varphi\colon Q &\rightarrow & F
\\
\beta &\mapsto & (\beta,0,\beta^{(p)}),
\end{eqnarray*}
for all $\beta\in Q$,
can be extended to $\tilde{\varphi}$. So, $\varphi$ is a morphism between $Q$ and $F$. Trivially $\varphi$ is injective.

Now it remains to see that $\varphi$ is surjective. First, note that
\[
F=\mathrm{conv}(\gamma\in {\mathrm{Cut}}^{\square}(G)\subseteq \RR^{|E|}:{\gamma}_{|E'|+1}=0).
\]
Thus, it is enough to see that for any vertex $\gamma$ of  ${\mathrm{Cut}}^{\square}(G)$ with ${\gamma}_{|E'|+1}=0$, there exists an element $\beta\in Q$ such that $\varphi(\beta)=\gamma$. Since such a vertex is indeed a cut vector of $G$, it follows that the corresponding partition $A|A^c$ of $V$ has the property that $u$ and $v$ both belong either to $A$ or to $A^c$. Without loss of generality assume that $u,v\in A$.

Note that for each $i=1,\ldots,p$, the two edges $\{u,v_i\}$ and $\{v,v_i\}$ either both belong to $\mathrm{Cut}(A)$ or both do not belong to this set.
By setting
\[
A':=(A\setminus \{u,v\})\cup \{w\},
\]
we get $A'|A^c$ as a partition of $V'$. Then, this defines a cut vector $\beta\in \RR^{|E'|}$ and we have $\varphi(\beta)=(\beta,0,\beta^{(p)})=\gamma$, as desired.
\end{proof}

As a consequence of the latter result we get the following face retracts for cut algebras:

\begin{Corollary}\label{contraction-retract}
Let $G$ be a graph and let $G'$ be a graph obtained by an edge contraction from $G$. Then $\KK[G']$ is a face retract of $\KK[G]$.
\end{Corollary}

\begin{proof}
Here we use the notation used in the proof of Proposition~\ref{face-contraction}, where we showed that  ${\mathrm{Cut}}^{\square}(G')$ is affinely isomorphic to a face $F$ of ${\mathrm{Cut}}^{\square}(G)$ via the aforementioned map $\varphi$. This isomorphism induces the following isomorphism between the corresponding polytopal algebras:
\begin{eqnarray*}
\label{isopoly}
\hat{\varphi} \colon \KK[{\mathrm{Cut}}^{\square}(G')]&\rightarrow& \KK[F]\\
{\bf{y}}^{\beta}z&\mapsto & {\bf{y}}^{(\beta,0,\beta^{(p)})}z,
\end{eqnarray*}
for all $\beta\in {\mathrm{Cut}}^{\square}(G')$.

On the other hand, according to \eqref{K[F]}, $\KK[F]$ is a face retract of $\KK[{\mathrm{Cut}}^{\square}(G)]$ which implies with the isomorphism $\hat{\varphi}$, that $\KK[{\mathrm{Cut}}^{\square}(G')]$ is a face retract of $\KK[{\mathrm{Cut}}^{\square}(G)]$ as well.

Finally, by applying the isomorphism given in Lemma~\ref{cut polytopal algebra}, we get that $\KK[G']$ is an algebra retract of $\KK[G]$.
\end{proof}

As examples, we have that $\KK[C_{n-1}]$ and $\KK[K_{n-1}]$ are algebra retracts of $\KK[C_n]$ and $\KK[K_n]$, respectively.

Finally, we discuss a similar operation as edge contraction on graphs. Indeed, after having seen this nice property of the edge contraction, one may think of the same property for another operation called \emph{vertex identification} or \emph{vertex contraction}. In vertex identification of a graph $G=(V,E)$, two different (not necessarily adjacent) vertices $u$ and $v$ are merged into a new vertex $w$ and a new graph $G'=(V',E')$ is provided with
\[
V'=(V\setminus \{u,v\})\cup \{w\}, ~~~~~ \mathrm{and}
\]
\[
E'=\big{\{}e\in E:e\cap N_G(\{u,v\})=\emptyset\big{\}}\cup \big{\{}\{w,z\}:z\in V'\cap N_G(\{u,v\})\big{\}}.
\]
In this case also we remove any possible resulting multiple edges from the obtained graph to get a simple graph. In the case that $u$ and $v$ are adjacent, the vertex contraction of $G$ is the same as the contraction of the edge $\{u,v\}$.

Now, let $G$ and $G'$ be as above. Then the cut algebra $\KK[G']$ is not necessarily an algebra retract of $\KK[G]$. For example, $K_4$ can be obtained by a vertex identification from the graph $G_6$ which is shown in Figure~\ref{fig}. But, $\KK[K_4]$ is not an algebra retract of $\KK[G_6]$, because $\KK[K_4]$ is a complete intersection while $\KK[G_6]$ is not (see Theorem~\ref{complete intersection}).

\section{Retracts of cut algebras from induced subgraphs}\label{induced}

In this section, we discuss the problem under which assumption an induced subgraph of a graph $G$ can provide an algebra retract of the cut algebra of $G$.

Note that an induced subgraph of a graph is also a minor of it which can not be obtained only by edge contractions. For example, $C_4$ has $P_3$ an induced subgraph which can not be obtained only by applying some edge contractions. We come back to this specific example in more details in the sequel.

It was stated in \cite[Lemma~3.1~(1)]{SS} that the cut polytope of an induced subgraph of a graph is (affinely isomorphic to) a face of the cut polytope of the original graph. This statement is not true in this generality, as will be discussed in the following. We are grateful to Hidefumi Ohsugi for showing us the next example.

\begin{Example}\label{Ohsugi}
{\em
The cut polytope of $P_3$ is a square.
The cut polytope of $C_4$ is the $4$-dimensional
crosspolytope, by \cite[Example~1.2]{SS}, which, e.g., also can be deduced from \cite[Proposition~1.7]{O2}.
Observe that crosspolytopes are simplicial. In particular, all
$2$-faces of $\mathrm{Cut}^\square(C_4)$ are triangles, and the square $\mathrm{Cut}^\square(P_3)$ can not be (affinely isomorphic to) a face of $\mathrm{Cut}^\square(C_4)$.
}
\end{Example}

Hence, the cut algebra of an induced subgraph of a graph does not provide necessarily a face retract. But, as we show in the sequel, there are still large classes of induced subgraphs which provide algebra retracts. First, we need the following definition.

\begin{Definition}\label{neighborhood}
{\em Let $G=(V,E)$ be a graph where $V=W\cup W'$ with $W\cap W'=\emptyset$ and $W,W'\neq \emptyset$ and let $H=G_{W}$. Suppose that there exists a vertex $v\in W$ with
$W\cap N_{G}(W')\subseteq N_{H}[v]$. Then we say that $H$ is a \emph{neighborhood-minor} of $G$. }
\end{Definition}

Note that any neighborhood-minor of a graph $G$ is an induced subgraph, and hence a minor of $G$. In addition, note that in the above example, $P_3$ is a neighborhood-minor of $C_4$.

\begin{Remark}\label{lifting}
{\em Despite the notion of neighborhood-minor as defined in this paper is not a classical notion in graph theory, in the special case $|W'|=1$ it has been previously considered in studying cut polytopes for different purposes. Namely the special case occurred in the study of lifting the defining inequalities of facets of an induced subgraph to those of the original graph (see, e.g., \cite[Theorem~2]{D}). }
\end{Remark}

The next theorem is one of the main results of this paper.

\begin{Theorem}\label{golden1}
Let $G$ be a graph, and let $H$ be a neighborhood-minor of $G$. Then $\KK[H]$ is an algebra retract of $\KK[G]$.
\end{Theorem}

\begin{proof}
Let $V=W\cup W'$ be the set of vertices of $G$ where $W\cap W'=\emptyset$ and $W,W'\neq \emptyset$, and such that $H=G_{W}$. Since $H$ is a neighborhood-minor of $G$, there exists a vertex $v\in W$ with $W\cap N_{G}(W')\subseteq N_{H}[v]$.

First we define a natural embedding of $S_{H}$ into $S_G$. Let $A|A^c$ be a partition of $W$, where we may assume that $v\in A$. Then $A\cup W'|A^c$ is a partition of $V$. Thus, by mapping $q_{A}$ to $q_{A\cup W'}$, we get a natural embedding $\lambda$ of $S_{H}$ into $S_G$, which induces the homogeneous homomorphism of $\KK$-algebras
\[
\overline{\lambda} \colon S_{H}/I_{H}\longrightarrow S_G/I_G.
\]
It is enough to  show that $\overline{\lambda}$ is well-defined. For this, we need to prove that
\[
\lambda(I_{H})\subseteq I_G.
 \]

Let $f=\prod_{i=1}^dq_{A_i}-\prod_{i=1}^dq_{C_i}$ be a generator of $I_{H}$ for some $d\in \mathbb{N}$, where we may assume that $v\in A_i$ and $v\in C_i$ for all partitions $A_i|{A_i}^c$ and $C_i|{C_i}^c$ of $W$. We have that $\phi_{H}(f)=0$, which implies by the definition of $\phi_{H}$ and using the bi-grading of $\KK[H]$ that
\begin{equation}\label{union cut}
\bigcup_{i=1}^d\mathrm{Cut}(A_i)=\bigcup_{i=1}^d\mathrm{Cut}(C_i).
\end{equation}
We also know
\[
\lambda(f)=\prod_{i=1}^dq_{A_i\cup W'}-\prod_{i=1}^dq_{C_i\cup W'}.
\]
It suffices to show that $\phi_G(\lambda(f))=0$ which is by definition of the map $\phi_G$ and using the bi-grading of $\KK[G]$ equivalent to the fact that
\[
\bigcup_{i=1}^d\mathrm{Cut}(A_i\cup W')=\bigcup_{i=1}^d\mathrm{Cut}(C_i\cup W').
\]

Suppose that $e=\{\ell,k\}\in \mathrm{Cut}(A_i\cup W')$ for some $i\in\{1,\ldots,d\}$, and assume that $k\in A_i\cup W'$ and $\ell\in A^c_i$. We distinguish the following two cases:

{\bf Case~(1).} Suppose that $k\in A_i$. Then $e\in \mathrm{Cut}(A_i)$ which implies by (\ref{union cut}) that $e\in \mathrm{Cut}(C_j)$ for some $j\in \{1,\ldots,d\}$. It is clear that $\mathrm{Cut}(C_j)\subseteq \mathrm{Cut}(C_j\cup W')$. Thus, it follows that $e\in \bigcup_{i=1}^d\mathrm{Cut}(C_i\cup W')$.

{\bf Case~(2).} Suppose that $k\in W'$. Note that $\ell\in W$, since
$\ell\in A^c_i\subseteq W$. Therefore $\ell\neq v$, and $\ell\in N_{H}[v]$, because by assumption we have $W\cap N_{G}(W')\subseteq N_{H}[v]$. It follows that $\ell \in N_{H}(v)\cap A^c_i$ which implies that $\{v,\ell\}\in \mathrm{Cut}(A_i\cup W')$. By Case~(1), we have $\{v,\ell\}\in \mathrm{Cut}(C_j\cup W')$ for some $j\in \{1,\ldots,d\}$. Thus, $e=\{k,\ell\}\in \mathrm{Cut}(C_j\cup W')$ as well, and hence
$e\in \bigcup_{i=1}^d\mathrm{Cut}(C_i\cup W')$.

We get
$\bigcup_{i=1}^d\mathrm{Cut}(A_i\cup W')\subseteq \bigcup_{i=1}^d\mathrm{Cut}(C_i\cup W')$.
The other inclusion follows similarly by symmetry. Hence it follows that $\overline{\lambda}$ is well-defined, as desired.

Let $\overline{\phi}_{H}$ and $\overline{\phi}_{G}$ be the isomorphisms induced by $\phi_{H}$ and $\phi_{G}$, respectively. The map  $\overline{\lambda}$ together with these two isomorphisms gives a homogeneous homomorphism of $\KK$-algebras
\begin{eqnarray*}
\iota \colon \KK[H] &\longrightarrow& \KK[G]
\\
u_A &\mapsto& u_A \prod_{\substack{i\in W', j\in N_G(W')\cap A^c\\
\mathrm{or}\\
j\in W', i\in N_G(W')\cap A^c}}s_{ij}
\prod_{\substack{i\in W', j\in W'\cup (N_G(W')\cap A)\\
\mathrm{or}\\
j\in W', i\in W'\cup (N_G(W')\cap A)}}t_{ij}.
\end{eqnarray*}
Thus, we obtain the following commutative diagram:
\begin{displaymath}\label{diagram}
\xymatrix{
S_{H}/I_{H} \ar[r]^{\overline{\lambda}} \ar@{->}[d]^{\overline{\phi}_{H}} &
S_G/I_G
\ar@{->}[d]^{\overline{\phi}_{G}}\\
\KK[H] \ar[r]^{\iota} & \KK[G] }.
\end{displaymath}

Next we show that $\iota$ is injective. Let $u_{A_1},\ldots u_{A_p}$ be all the monomial generators of the $\KK$-algebra $\KK[H]$, and let $g$ be any polynomial in the polynomial ring $\KK[x_1,\dots,x_p]$ such that $\iota(g(u_{A_1},\ldots u_{A_p}))=0$ in $\KK[G]$. In the latter equality, we put $s_{ij}=t_{ij}=1$ for all $i,j\in V$ such that $\{i,j\}\cap W'\neq \emptyset$. This implies by definition of $\iota$ that $g(u_{A_1},\ldots u_{A_p})=0$. Hence $\iota$ is injective.

Note that if $A|A^c$ is a partition of $V$, then $A\cap W|A^c\cap W$ is clearly a partition of $W$. So, we define the homogeneous homomorphism of $\KK$-algebras
\begin{eqnarray*}
\pi \colon \KK[G]&\longrightarrow& \KK[H]
\\
u_{A}&\mapsto& u_{A\cap W}.
\end{eqnarray*}

It is enough to see that $\pi$ is well-defined. Suppose that $u_{A_1},\ldots,u_{A_r}$ are all the monomial generators of the $\KK$-algebra $\KK[G]$, and let $h$ be any polynomial in the polynomial ring $\KK[x_1,\ldots,x_r]$, with $h(u_{A_1},\ldots,u_{A_r})=0$. Then by putting $s_{ij}=t_{ij}=1$ for all $i,j$ such that $\{i,j\}\cap W'\neq \emptyset$ in the last equality, we get $h(\pi(u_{A_1}),\ldots , \pi(u_{A_r}))=0$.

Finally we need to check that $\pi\circ \iota=\mathrm{id}_{\KK[H]}$ which implies that $\KK[H]$ is an algebra retract of $\KK[G]$, and hence the desired result of the theorem follows. It suffices to check the latter equality for any monomial generator of the $\KK$-algebra $\KK[H]$, i.e. $u_{A}$ where $A\subseteq W$. But, this easily holds, since $u_{A}$ is not divided by any variables $s_{ij}$ or $t_{ij}$ with $\{i,j\}\cap W'\neq \emptyset$.
\end{proof}

\begin{Remark}\label{nonface retract}
{\em We would like to remark that the algebra retract given in Theorem~\ref{golden1} is not a face retract. For this observe at first that a face retract of a cut algebra is always $\ZZ^{2|E|}$-homogeneous. Now let $G=C_4$ and $H=P_3$ with $V(G)=\{1,2,3,4\}$ and $V(H)=\{1,2,3\}$. Then one can observe that the map $\iota$ is not $\ZZ^{2|E|}$-homogeneous. Indeed, in this case $t_{12}t_{23}t_{34}t_{14}-t_{12}t_{23}s_{34}s_{14}$ is an element of $\ker \pi$ which is not $\ZZ^{2|E|}$-homogeneous. Here $\iota$ and $\pi$ are as in the proof of Theorem~\ref{golden1},
}
\end{Remark}

\begin{Remark}\label{natural embedding}
{\em One may ask if there exists an induced subgraph $H$ of a graph $G$ such that $\KK[H]$ is not an algebra retract of $\KK[G]$. So far we do not know such examples by comparing their algebraic invariants and properties preserved by retracts. But, there are examples for which there is no algebra retract induced by an embedding of $S_H$ into $S_G$ via mapping the variables of $S_H$ to the ones of $S_G$ as in the proof of Theorem~\ref{golden1}. The graph $G_{10}$ depicted in Figure~\ref{fig} together with its induced subgraph $C_4$ is such an example.
}
\end{Remark}

In the following we consider some special cases of neighborhood-minors of a graph.

In \cite{SS}, in particular, a \emph{clique-sum} $G_1\sharp G_2$ of two graphs $G_1$ and $G_2$ was studied. More generally, an \emph{$H$-sum} of $G_1$ and $G_2$, along an induced subgraph $H$ of both of them, is defined as follows:

\begin{Definition}\label{H-sum}
{\em Let $G_1=(V_1,E_1)$, $G_2=(V_2,E_2)$ be two graphs and assume that ${(G_1)}_{V_1\cap V_2}={(G_2)}_{V_1\cap V_2}$ which we denote by $H$. Then an \emph{$H$-sum} $G_1\sharp_{H} G_2$ of $G_1$ and $G_2$, along the induced subgraph $H$, is defined to be the graph with
\[
V(G_1\sharp_{H} G_2)=V_1\cup V_2 ~~~~~~ \mathrm{and} ~~~~~~ E(G_1\sharp_{H} G_2)=E_1\cup E_2.
\]
}
\end{Definition}

In particular, if $H$ is a clique of $G_1$ and $G_2$, then the definition of an $H$-sum is the same as the one of a clique-sum.

We would like to mention that this operation on two graphs is sometimes called \emph{gluing}, and then the common induced subgraph $H$ of the two graphs is sometimes called a \emph{clone}. Some authors also use other notation for this, but here we follow the setting used in earlier papers on cut ideals.

We would also like to remark that by fixing a labeling for $G_1$, $G_2$ and $H$, the $H$-sum $G_1\sharp_{H} G_2$ is uniquely determined. In general, this is not true even up to isomorphism, though it might be true in special cases.

For example, in the case of the graphs $K_2$ and $K_3$ with $H=K_1$, the $H$-sum is unique up to isomorphism of graphs. Therefore, we use the notation $K_2\sharp_{K_1} K_3$. But if we consider the graphs $K_3$ and $P_3$ with $H=K_1$, then both of the graphs $G_3$ and $G_4$ in Figure~\ref{fig} are an $H$-sum of $K_3$ and $P_3$. Hence the notation $K_3\sharp_{K_1} P_3$ can not be used. In the sequel, we use this notation only in the case where there exists a uniquely determined $H$-sum (up to isomorphism).

\begin{Corollary}\label{glue}
Let $G_1$ and $G_2$ be two graphs on $V_1$ and $V_2$, respectively, and let $H={(G_1)}_{V_1\cap V_2}={(G_2)}_{V_1\cap V_2}$ with $|V_1\cap V_2|\geq 1$. If $H$ has a vertex $v$ with $\deg_{H}(v)=|V_1\cap V_2|-1$, then $G_i$ is a neighborhood-minor of $G$, and hence $\KK[G_i]$ is an algebra retract of $\KK[G_1\sharp_{H} G_2]$ for $i=1,2$.

In particular, for $i=1,2$, $\KK[G_i]$ is an algebra retract of the cut algebra of any clique-sum of $G_1$ and $G_2$.
\end{Corollary}

\begin{proof}
Let $G:=G_1\sharp_{H} G_2$, $W:=V_1$ and $W':=V_2\setminus V_1$.
Note that $v\in W$, and by definition of $G$, we have $N_{G}(W')\subseteq V_2$ and $V_1\cap V_2\subseteq N_{G_1}[v]$, where the latter inclusion  holds since $\deg_{H}(v)=|V_1\cap V_2|-1$. Altogether we have that
\[
W\cap N_G(W')\subseteq V_1\cap V_2\subseteq N_{G_1}[v],
\]
and hence $G_1$ is a neighborhood-minor of $G$. Thus, by Theorem~\ref{golden1}, we have that $\KK[G_1]$ is an algebra retract of $\KK[G]$, as desired. Similarly, $G_2$ is also a neighborhood-minor of $G$, and hence $\KK[G_2]$ is an algebra retract of $\KK[G]$ as well.

The second assertion follows, since in a clique-sum of $G_1$ and $G_2$, the graph $H$ is indeed a clique, say on $k$ vertices, whose all vertices have clearly degree $k-1$ in $H$.
\end{proof}

A straightforward consequence of Theorem~\ref{golden1} is the following corollary. For this we recall before the definition of a vertex duplication in a graph. A \emph{duplication} of a vertex $v$ of a graph $G=(V,E)$ produces a graph $G'=(V',E')$ by adding a new vertex $v'$ such that
\[
V'=V\cup \{v'\}
\]
and
\[
E'=E\cup \big{\{}\{u,v'\}: u\in N_G(v) \big{\}}.
\]
Note that $N_{G'}(v') = N_{G'}(v)=N_G(v)$.

\begin{Corollary}\label{duplicate}
Let $G$ be a graph on $V$, and let $u,v\in V$ such that $u\neq v$ and $N_{G}(u)\subseteq N_{G}[v]$. Then ${G}_{V\setminus \{u\}}$ is a neighborhood-minor of $G$, and hence $\KK[{G}_{V\setminus \{u\}}]$ is an algebra retract of $\KK[G]$. In particular, if $G'$ is obtained by a duplication of a vertex of $G$, then $\KK[G]$ is an algebra retract of $\KK[G']$.
\end{Corollary}

\begin{proof}
Let $W:=V\setminus \{u\}$ and $W':=\{u\}$. Then we have
\[
W\cap N_G(W')=(V\setminus \{u\})\cap N_{G}(u)\subseteq (V\setminus \{u\})\cap N_{G}[v]=N_{G_W}[v].
\]
Hence ${G}_{V\setminus \{u\}}$ is a neighborhood-minor of $G$. Thus, by Theorem~\ref{golden1} the result follows.
\end{proof}

In the following we discuss some well-known families of graphs whose certain induced subgraphs give algebra retracts.

\begin{Example}\label{Different classes}
\begin{enumerate}
\em{\item {\em{Chordal graphs}}. Let $G$ be a chordal graph (i.e. a graph whose induced cycles have length $3$). Then by Dirac's theorem in \cite{Di} (see also, e.g., \cite[Lemma~9.2.6]{HH}), the facets of the clique complex $\Delta(G)$ of $G$ can be ordered, say  $F_1,\ldots,F_r$, such that for each $i=1,\ldots,r$, the facet $F_i$ is a leaf of the simplicial complex $\left\langle F_1,\ldots,F_i\right\rangle$. Recall that by a \emph{leaf} of a simplicial complex $\Delta$, we mean a facet $F$ such that either it is the only facet of $\Delta$ or there exists a facet $G\neq F$ of $\Delta$ with $H\cap F\subseteq G\cap F$ for all facet $H\neq F$ in $\Delta$. Such an order of facets of a simplicial complex is called a \emph{leaf order}, and such a simplicial complex is called \emph{quasi-forest} (see \cite{Z}). Now, let $\Delta_i:=\left\langle F_1,\ldots,F_i\right\rangle$ for all $i=1,\ldots,r$. Then we have $\Delta_i=\Delta(G_i)$ which is the clique complex of the induced subgraph  $G_i$ of $G$ on the vertex set $\cup_{j=1}^iF_j$. In particular, $G_r=G$. It is also clear that $G_i$ is an induced subgraph of $G_{i+1}$ for all $i=1,\ldots,r-1$. Since $F_i$ is a leaf of $\Delta_i$, there exists some $j\in \{1,\ldots,i-1\}$ such that $F_i\cap F_{\ell}\subseteq F_j\cap F_i$ for all $\ell=1,\ldots,i-1$.
Then it follows that for all $i$, $G_{i+1}$ is indeed the clique-sum of $G_{i}$ and the complete graph with the vertex set $F_i$ over the clique $F_j\cap F_i$. Therefore, by Corollary~\ref{glue}, for all $i=1,\ldots,r-1$, the algebra $\KK[G_i]$ is an algebra retract of $\KK[G_{i+1}]$, and hence of $\KK[G]$. In particular, since trees are chordal graphs, by removing a vertex of degree one from a tree, we always obtain an algebra retract of the cut  algebra associated to the original tree.

\item {\em{Complete $t$-partite graphs}}. Let $G$ be a complete $t$-partite graph on the vertex set $V$ with the partition $V_1,\ldots,V_t$. We consider the following possibilities to determine retracts:
\begin{enumerate}
\item Let $v\in V_k$ for some $k\in \{1,\dots,t\}$. By definition, $N_G(v)=V\setminus V_k$. Let $H$ be an induced subgraph of $G$ which is also complete $t$-partite, with the partition $W_1,\ldots,W_t$ of its vertex set where $W_i$ is a nonempty subset of $V_i$ for each $i$. Then it follows from Corollary~\ref{duplicate} that $\KK[H]$ is an algebra retract of $\KK[G]$.

\item Let $H$ be an induced subgraph of $G$ which is a complete $s$-partite graph with the partition $W_1,\ldots,W_s$ of its vertex set, where
$\emptyset\neq W_i\subseteq V_i$ for each $i=1,\ldots,s$, and
$1\leq s < t$. Assume in addition that $|W_i|=1$ for some $i\in \{1,\ldots,s\}$. We claim that $\KK[H]$ is an algebra retract of $\KK[G]$.

Indeed, let $w$ be the only element of $W_i$. Then we choose an element $v$ of $V_{s+1}$, and consider the induced subgraph $H_1$ of $G$ on the vertex set $\{v\}\cup (\bigcup_{i=1}^s W_i)$ which is a complete $(s+1)$-partite graph. Since $N_{H_1}(v)=N_{H_1}[w]$, Corollary~\ref{duplicate} implies that $\KK[H]$ is an algebra retract of $\KK[H_1]$. By repeating this procedure, we get after $t-s$ steps, a sequence of graphs, and hence algebra retracts. In fact, we get that $\KK[H_i]$ is an algebra retract of $\KK[H_{i+1}]$ for $i=1,\ldots,t-s-1$,  where $H_i$ is a complete $(s+i)$-partite induced subgraph of $G$. By the above discussion in~(a), the complete $t$-partite subgraph $H_{s-t}$ of $G$ provides an algebra retract $\KK[H_{s-t}]$ of $\KK[G]$. Hence, $\KK[H]$ is an algebra retract of $\KK[G]$.
\end{enumerate}

\item {\em{Ferrers graphs}}. First recall that a Ferrers graph $G$ is a bipartite graph whose vertex set $V$ can be partitioned into $X=\{x_1,\ldots,x_n\}$ and $Y=\{y_1,\ldots,y_m\}$ such that $\{x_1,y_m\}$
and $\{x_n,y_1\}$ are edges of $G$, and in addition, if $\{x_i,y_j\}$ is an edge of $G$, then so is $\{x_p,y_q\}$ for
$1\leq p\leq i$ and $1\leq q\leq j$. In particular, a complete bipartite graph is a Ferrers graph. We claim that for any induced subgraph $H$ of a Ferrers graph $G$, the algebra $\KK[H]$ is an algebra retract of $\KK[G]$.

Since by a relabeling of the vertices (if needed), any induced subgraph of $G$ is again a Ferrers graph, it is enough to consider an induced subgraph of $G$ which has only one vertex less than $G$. Without loss of generality, assume that $H$ is an induced subgraph of $G$ on the vertex set $V\setminus \{y_k\}$ for some $k\in \{1,\ldots,m\}$. Since by definition, $\{x_1,y_{\ell}\}$ is an edge of $G$ for all $\ell=1,\ldots,m$, it follows that $N_G(y_k)$ is a non-empty subset of $X$. Let $t$ be the greatest index such that $x_t\in N_G(y_k)$, which implies by definition of a Ferrers graph that
$N_G(y_k)=\{x_1,\ldots,x_t\}$. Since $\{y_1,x_n\}$ is an edge of $G$, it follows that $\{y_1,x_i\}$ is an edge of $G$ for all $i=1,\ldots,n$, and hence $N_G[y_1]=X$. Thus $N_G(y_k)\subseteq N_G[y_1]$, and using  Corollary~\ref{duplicate}, we get that $\KK[H]$ is an algebra retract of $\KK[G]$.

\item {\em{Ring graphs}}. A ring graph is a graph obtained from cycles and trees by clique-sums along subgraphs isomorphic to $K_1$ and $K_2$, (see, e.g., \cite{GRV1,GRV2,NP}). In particular, cycles and trees are ring graphs.

So, if $G$ is a ring graph, then it is the clique-sum of a ring graph $H_1$ and a cycle or a tree, say $H_2$, i.e.  $G=H_1\sharp_{K_i} H_2$ where $i=1$ or $i=2$. Hence, $\KK[H_1]$ and $\KK[H_2]$ are both algebra retracts of $\KK[G]$ by Corollary~\ref{glue}.
}
\end{enumerate}
\end{Example}

\section{Applications}\label{application}

In this section, we apply the tools from the previous sections to discuss certain algebraic properties and invariants of cut ideals and algebras as some applications.

First, we define the following notion:

\begin{Definition}
{\em Let $G$ be a graph. Then:
\begin{enumerate}
\item [(a)] We say that a graph $G'$ is a \emph{combinatorial retract} of $G$, if there is a sequence $G_0,G_1,\ldots,G_r$ of graphs where $G_0=G$ and $G_r=G'$ and for each $i=1,\ldots,r$, $G_i$ is either
\begin{enumerate}
\item[(i)] obtained by an edge contraction from $G_{i-1}$, or
\item[(ii)] a neighborhood-minor of $G_{i-1}$.
\end{enumerate}
\item [(b)] Given other graphs $H_1,\ldots,H_t$ with $t\geq 1$, then we say that $G$ is \emph{$(H_1,\ldots,H_t)$-combinatorial retract-free} if it has no combinatorial retract isomorphic to any of $H_1,\ldots,H_t$.
\end{enumerate}
}
\end{Definition}

\begin{Remark}\label{combin retract K_n}
{\em Note that any combinatorial retract is a minor. Observe that for a graph $G$ having a neighborhood-minor isomorphic to a complete graph $K_n$ is equivalent to having $K_n$ as a minor. Indeed, if $G$ has such a minor, then it is obtained only by a sequence of edge contractions by Remark~\ref{contraction K_n}. So, it is a combinatorial retract of $G$. }
\end{Remark}

By Corollary~\ref{contraction-retract} and Theorem~\ref{golden1}, if $G'$ is a combinatorial retract of $G$, then $\KK[G']$ is an algebra retract of $\KK[G]$. As special cases see Example~\ref{Different classes}, where the mentioned induced subgraphs are combinatorial retracts of the given graph.

Next we discuss some algebraic properties of cut ideals. Following \cite{SS}, for a graph $G$, we let $\mu(I_G)$ be the highest degree of an element in a minimal generating system of $I_G$.

\begin{Corollary}\label{greatest generator}
Let $G$ be a graph and let $G'$ be a combinatorial retract of $G$. Then:
\begin{enumerate}
 \item[{\em (a)}] $\beta_{i,j}^{S_{G'}}(I_{G'})\leq \beta_{i,j}^{S_G}(I_G)$ for all $i,j$.
\item[{\em (b)}] $\reg_{S_{G'}} I_{G'}\leq \reg_{S_{G}} I_G$ and $\projdim_{S_{G'}} I_{G'}\leq \projdim_{S_{G}} I_G$.
\item[{\em (c)}] The number of minimal generators of a given degree $j$ of $I_{G'}$ does not exceed the ones of $I_G$.
\item[{\em (d)}] $\mu(I_{G'})\leq \mu(I_G)$.
\end{enumerate}
\end{Corollary}

\begin{proof}
Parts (a) and (b) are just immediate consequences of Proposition~\ref{Betti}, Corollary~\ref{contraction-retract} and Theorem~\ref{golden1}.

Part (c) just follows from (a) by considering the case $i=0$ and (d) is a trivial consequence from this.
\end{proof}

In \cite[Corollary~3.3~(2)]{SS}, the authors discussed the inequality of part~(d) in Corollary~\ref{greatest generator} in the case of contracting an edge which is a particular case here.
The proof of the induced subgraph case in \cite[Corollary~3.3~(1)]{SS}
is no longer valid since it is based on \cite[Lemma~3.2~(1)]{SS} which is not correct (see Example~\ref{Ohsugi}). But in certain special cases one can deduce the inequality; see, e.g., Theorem~\ref{golden1}.

The following corollary is an immediate consequence of Corollary~\ref{greatest generator} which verifies a weaker version of  \cite[Conjecture~3.1]{SS} where it was conjectured that the set of graphs $G$ with $\mu(I_G)\leq k$ is minor-closed for any $k$.

\begin{Corollary}\label{combin retract closed}
The set of graphs $G$ with $\mu(I_G)\leq k$ is combinatorial retract-closed for any $k$.
\end{Corollary}

As another consequence of Corollary~\ref{greatest generator}, we can classify those connected graphs whose cut ideals are generated in a single degree.

\begin{Corollary}\label{single degree}
Let $G$ be a connected and non-complete graph. Then $I_G$ has a generator of degree~$2$. In particular, the following statements are equivalent:
\begin{enumerate}
\item [{\em (a)}] $I_G$ is generated in a single degree;
\item [{\em (b)}] $I_G$ is generated in degree $2$;
\item [{\em (c)}] $G$ is $K_4$-minor-free.
\end{enumerate}
\end{Corollary}

\begin{proof}
Since $G$ is non-complete and connected, it follows that it has an induced subgraph isomorphic to $P_3$, say on the set of vertices $\{u,v,w\}$. Without loss of generality, we denote that induced subgraph by $P_3$, and also assume that $\deg_{P_3}(v)=2$. Therefore, it follows from Corollary~\ref{glue} that the induced subgraph $P_3$ is indeed a neighborhood-minor of $G$. Thus, by Corollary~\ref{greatest generator}~(c), we deduce that $I_G$ has a minimal generator of degree $2$, because $I_{P_3}$ is generated in degree $2$ by \cite[Corollary~4.3]{NP} (see also \cite[Example~2.3]{SS}).
Hence, (a) and (b) are immediately equivalent. On the other hand, (b) and (c) are also equivalent, as it was already shown in \cite[Corollary~2.8]{E}.
\end{proof}

\begin{Remark}\label{single degree K_n}
{\em Using computer algebra systems one can extend the characterization in Corollary~\ref{single degree} as follows. There we assumed that $G$ is not a complete graph. If $G=K_2$ or $G=K_3$, then $I_G=\langle 0 \rangle$, by Proposition~\ref{regularity0}.
If $G=K_4$, we see in Example~\ref{K_4} that its cut ideal is a principal ideal generated in degree $4$. If $G=K_5$, then computations show that the cut ideal is generated in degrees $4$ and $6$. Since $K_5$ is a neighborhood-minor of $K_n$ for all $n\geq 6$, it follows by Proposition~\ref{single degree}~(c) that $I_{K_n}$ is not generated in a single degree for all $n\geq 5$. So, together with Corollary~\ref{single degree}, one can get the complete characterization of all connected graphs whose cut ideals are generated in a single degree. More precisely, this is the case if and only if $G=K_4$ or $G$ is $K_4$-minor-free. }
\end{Remark}

Now we consider the disconnected case. First we need to recall a result of \cite{NP} which deals with the cut algebra of disconnected graphs.

\begin{Proposition}
{{\em {(}}\cite[Proposition~5.2]{NP}{\em {)}}}\label{disconnected}
 Let $G=G_1\sqcup G_2$ where $G_i=(V_i,E_i)$ is a graph with $|V_i|=n_i$ for $i=1,2$. Then there is an injective homogeneous $\mathbb{K}$-algebra homomorphism $\alpha \colon S_{G_1\sharp_{K_1} G_2}\rightarrow S_{G}$ mapping variables of $S_{G_1\sharp_{K_1} G_2}$ to the variables of $S_{G}$ such that $I_G=\alpha(I_{G_1\sharp_{K_1} G_2})+L$,
where $L$ is an ideal minimally generated by $2^{n_1+n_2-2}$ linear forms. In particular, $\mathbb{K}[G]\iso \mathbb{K}[G_1\sharp_{K_1} G_2]$ where $K_1$ can be any vertex of $G_1$ and $G_2$, respectively.
\end{Proposition}

In the following proposition, we determine when the cut ideal of a graph is generated only by linear forms:

\begin{Proposition}\label{single degree disconnected}
Let $G=(V,E)$ be a disconnected graph with $|E|\geq 1$. Then the following statements are equivalent:
\begin{enumerate}
\item [{\em (a)}] $I_G$ is generated in a single degree;
\item [{\em (b)}] $I_G$ is generated by linear forms;
\item [{\em (c)}] $G=K_2\sqcup(\sqcup_{i=1}^t K_1)$ or $G=K_3\sqcup(\sqcup_{i=1}^t K_1)$ for some $t\geq 1$.
\end{enumerate}
\end{Proposition}

\begin{proof}
Here we use the notation in Proposition~\ref{disconnected}. By applying repeatedly the same proposition, we get
\[
\alpha(I_{K_i\sqcup (\sqcup_{i=1}^{t-1} K_1)\sharp_{K_1} K_1} )=\alpha(I_{K_i})=\langle 0 \rangle
\] for $i=2,3$. So by Proposition~\ref{disconnected}, (c) implies (b). Obviously, (b) implies (a). Thus, it remains to prove that (a) implies (c). Since by Proposition~\ref{disconnected} the ideal $I_G$ has always linear forms among the generators, being generated in a single degree, clearly means that $I_G$ is generated only by linear forms. Let $G=\sqcup_{i=1}^cG_i$ where $c\geq 1$ and $G_i$ is a connected component of $G$ for $i=1,\ldots,c$.
If $c=2$, then $G_1\sharp_{K_1} G_2$ is connected. So $\alpha(I_{G_1\sharp_{K_1} G_2})=\langle 0 \rangle$ by Proposition~\ref{linear forms}. Since $\alpha$ is injective, it follows that $I_{G_1\sharp_{K_1} G_2}=\langle 0 \rangle$. Thus Proposition~\ref{regularity0} implies that $G_1\sharp_{K_1} G_2=K_2$ or $G_1\sharp_{K_1} G_2=K_3$, which is the case if and only if $G_1=K_i$ for $i\in\{2,3\}$ and $G_2=K_1$, or vice versa. Then by induction on the number of connected component of $G$, namely $c$, it follows that (a) implies (c), as desired.
\end{proof}

In the sequel, first we study the property of being a complete intersection (which is preserved under algebra retracts) for cut algebras. The obtained results will then be applied to study the resolution of cut ideals later.

There is a conjecture in \cite{SS} on the characterization of Cohen-Macaulay cut algebras which has been studied in some special cases, (see, e.g., \cite{NP} for ring graphs). In \cite{O2} those graphs whose cut algebras are normal and Gorenstein were characterized. In the following, we classify cut algebras with respect to being a complete intersection.

\begin{Theorem}\label{complete intersection}
Let $G=(V,E)$ be a graph with $|E|\geq 1$, which has no isolated vertices. Then the following statements are equivalent:
\begin{enumerate}
\item [{\em (a)}] $\KK[G]$ is a complete intersection;
\item [{\em (b)}] $G$ is one of the following graphs: $K_2, K_3, P_3, 2K_2, C_4, K_4\setminus e, K_4$.
\end{enumerate}
\end{Theorem}

\begin{proof}
Let $|V|=n$. First, we characterize all graphs $G$ with no isolated vertices, $|E|\geq 1$ and $n\leq 4$. (see Table~\ref{table:1}), for which $\KK[G]$ is a complete intersection.

By Proposition~\ref{regularity0} and Example~\ref{K_4}, $\KK[K_i]$ is a complete intersection for $i=2,3,4$.

If $G\neq K_2$ is a tree, namely $P_3, P_4, K_{1,3}$, then by comparing \eqref{height} and the formula given in \cite[Corollary~4.2]{NP} for the number of minimal generators of $I_G$, it follows that in this case, $\KK[G]$ is a complete intersection if and only if $G=P_3$. Moreover, by Proposition~\ref{disconnected}, the cut algebras of $P_3$ and $2K_2$ are isomorphic which implies that $\KK[2K_2]$ is also a complete intersection.

If $G=K_2\sharp_{K_1} K_3$, then by a criteria given in \cite[Theorem~3.4]{O2} for Gorenstein normal cut algebras, it follows that $\KK[G]$ is not Gorenstein, and hence it is not a complete intersection. Note that by \cite[Example~3.7]{O1} $\KK[G]$ is normal in this case.

Comparing the formula given in \cite[Proposition~3.7]{NP} for the number of minimal generators of the cut ideal of a cycle with the height of $I_{C_4}$, it follows that $\KK[C_4]$ is a complete intersection.

Finally, by Example~\ref{K_4} and Example~\ref{another small}~(d), $\KK[K_4]$ and $\KK[K_4\setminus e]$ are complete intersections.

In particular, this characterization already proves the implication~(b)\implies (a).

Next, we show that if $n=5$, then $\KK[G]$ is not a complete intersection which implies that the same result holds for $n>5$. In fact, by suitable  edge contractions of a graph with more than five vertices, a graph with five vertices is obtained, and hence the desired conclusion follows in this case, since we get indeed an algebra retract.

It remains to deal with the case $n=5$. We consider all graphs in Table~\ref{table:1}.
Note that by Example~\ref{K_5}, $\KK[K_5]$ is not Cohen-Macaulay, and hence not a complete intersection. Moreover, it is enough to show that  cut algebras of the following graphs, which are the only Gorenstein ones by \cite[Theorem~3.4]{O2}, are not complete intersections. Indeed, in each case, we obtain a combinatorial retract with four vertices whose cut algebra is not a complete intersection. Hence it follows that $\KK[G]$ is not a complete intersection as well.
\begin{itemize}
\item $P_5$, $K_{1,4}$ and $G_1$: in these cases, by contracting an edge, we obtain either $P_4$ or $K_{1,3}$.
\item $K_2\sharp_{K_1} C_4$: by contracting an edge of the induced $C_4$, we get $K_2\sharp_{K_1} K_3$.
\item $K_3\sharp_{K_1} K_3$: by contracting an edge, we obtain again $K_2\sharp_{K_1} K_3$.
\item $C_4\sharp_{P_3} C_4$: in this case, $K_{1,3}$ is neighborhood-minor. Note that in this case, contracting any of the edges of $C_4\sharp_{P_3} C_4$ yields graphs with complete intersection cut algebras.
\item $G_7$: by contracting the common edge in the two triangles, one gets $K_2\sharp_{K_1} K_3$.
\item $G_8$: by contracting the common edge between the three triangles, one obtains $K_{1,3}$.
\item $K_3 \sharp_{K_2} K_4$: by contracting the common edge between the induced $K_3$ and $K_4$, we get $K_2\sharp_{K_1} K_3$.
\item $K_5\setminus e$: in this case, by~\eqref{height}, $\height I_{K_5\setminus e}=6$, while the number of minimal generators of $I_{K_5\setminus e}$ is $35$ (see Example~\ref{another small}).
\end{itemize}
\end{proof}

Note that a disjoint union of any of the graphs in part~(b) of Theorem~\ref{complete intersection} with some isolated vertices, preserves the property of being a complete intersection, by Proposition~\ref{disconnected}.

In the next theorem, we determine all graphs whose cut ideals have  linear resolutions.

\begin{Theorem}\label{linear resolution}
Let $G=(V,E)$ be a connected graph with $|E|\geq 1$. Then the following statements are equivalent:
\begin{enumerate}
\item [{\em (a)}] $I_G$ has a $d$-linear resolution for some $d\geq 2$;
\item [{\em (b)}] $G=P_3$ or $G=K_2\sharp_{K_1} K_3$ or $G=K_4$.
\end{enumerate}
\end{Theorem}

\begin{proof}
The implication (b)\implies~(a) follows from Examples~\ref{K_4}~and~\ref{another small}, namely, $I_{P_3}$ and $I_{K_2\sharp_{K_1} K_3}$ have $2$-linear resolution, and $I_{K_4}$ has a $4$-linear resolution.

It remains to prove (a)\implies (b). Assume that $I_G$ has a $d$-linear resolution with $d\geq 2$. Then $I_G$ is generated in a single degree $d\geq 2$. On the one hand, by Proposition~\ref{regularity0}, it follows that $G\neq K_2,K_3$. On the other hand, by Corollary~\ref{single degree} and Remark~\ref{single degree K_n}, it follows that $G$ is either $K_4$ or $K_4$-minor-free. Now, we need to show that for any $K_4$-minor-free graph which is not isomorphic to $P_3$ or $K_2\sharp_{K_1} K_3$, the ideal $I_G$ does not have a linear resolution. Let $|V|=n$. We may assume that  $n\geq 4$, since the only connected graphs with $n<4$ and $|E|\geq 1$, are $K_2$, $K_3$ and $P_3$.

First suppose that $n=4$. Then according to Table~\ref{table:1}, we only need to consider $G=P_4, K_{1,3}, C_4, K_4\setminus e$. In the first two cases, which are trees with three edges, by \cite[Corollary~4.3]{NP} it follows that their cut ideals are generated in degree $d=2$. By \cite[Proposition~4.4]{NP} we have that $\reg_{S_G}(I_G)=4$, which implies that $I_G$ does not have a $2$-linear resolution. In the cases where $G=C_4$ or $G=K_4\setminus e$, we know, by Theorem~\ref{complete intersection} and Proposition~\ref{projective dimension}, that $S_G/I_G$ is a complete intersection with projective dimension~$3$ or $2$. This implies that $\beta_{1,4}^{S_G}(I_G)\neq 0$, since in both cases $G$ is a ring graph, and hence by \cite[Theorem~6.2]{NP} $I_G$ is generated by quadrics. Thus, we deduce that $I_G$ does not have a linear resolution.
Therefore, the only graphs with $n\leq 4$ whose cut ideals have linear resolutions are $P_3$, $K_2\sharp_{K_1} K_3$ and $K_4$.

Next, we show that for $n=5$, there is no $K_4$-minor-free graph $G$ for which $I_G$ has a linear resolution. Then, this implies by Corollary~\ref{greatest generator}~(a) that the cut ideal of a $K_4$-minor-free graph with more than $5$ vertices does not have linear resolution as well.

Indeed, this follows from the fact that from any connected graph with $n$ vertices, one obtains, as a combinatorial retract, a connected graph with five vertices after a sequence of edge contractions. Let $n=5$. We consider the graphs with notation in Table~\ref{table:1}.

One the one hand, note that if $G=K_2\sharp_{K_1} K_4, K_3\sharp_{K_2} K_4, K_5\setminus e, K_5$, then it is easily observed that by one edge contraction, one gets $K_4$. So that they are not $K_4$-minor-free. On the other hand, note that since the cut algebra of all graphs with $n=5$, except $K_5$, are normal by \cite[Example~3.7]{O1}, we can apply the characterization of normal Gorenstein cut algebra of graphs given in \cite[Theorem~3.2]{O2}. Since by that theorem, the cut algebras of the graphs $P_5, K_{1,4}, G_1, K_2\sharp_{K_1} C_4, K_3\sharp_{K_1} K_3, C_4\sharp_{P_3} C_4, G_7$ and $G_8$ (from Table~\ref{table:1}) are Gorenstein, their cut ideals do not have a linear resolution, since they are not principal ideals in those cases.

In the remaining cases, by contraction of an edge of the graph we obtain a graph with four vertices whose cut ideal does not have a linear resolution as we showed before. More precisely:

If $G=G_2, G_3, G_4$, then by contracting an edge of their unique induced triangle, one obtains some trees on four vertices whose cut ideals do not have linear resolutions as we showed before.

If $G=C_5$, then we get $C_4$, where $I_{C_4}$ does not have a linear resolution.

If $G=G_5, G_6$, then by contracting an edge incident to the vertex of degree $1$, one obtains $K_4\setminus e$ whose cut ideal does not have a linear resolution.

Finally, if $G=K_3\sharp_{K_2} C_4, G_9, G_{10}$, then for example by contracting an edge of a triangle of the graph which does not belong to the unique induced $C_4$, one gets either $C_4$ or $K_4\setminus e$, whose cut ideals do not have linear resolutions.
\end{proof}

Proposition~\ref{regularity0}, Proposition~\ref{linear forms} and Proposition~\ref{single degree disconnected} gave the characterization of all graphs $G$ with $\reg_{S_G} S_G/I_G=0$. Example~\ref{K_4} shows in particular that $I_{K_4}$ has regularity~$4$. Thus, Theorem~\ref{linear resolution} together with Proposition~\ref{linear forms} proves:

\begin{Corollary}\label{regularity1}
Let $G=(V,E)$ be a connected graph with $|E|\geq 1$. Then the following statements are equivalent:
\begin{enumerate}
\item [{\em (a)}] $I_G$ has a $2$-linear resolution;
\item [{\em (b)}] $\reg_{S_G} I_G=2$;
\item [{\em (c)}] $G=P_3$ or $G=K_2\sharp_{K_1} K_3$.
\end{enumerate}
\end{Corollary}

\begin{Remark}\label{reg2-disconnected}
{\em We would like to remark that if $G$ is a disconnected graph with connected components $G_1$ and $G_2$, then by using Proposition~\ref{disconnected} and \cite[Remark~2.1]{BCR}, one obtains that $\reg_{S_G} I_G=\reg_{S_{G_1\sharp_{K_1} G_2}} I_{G_1\sharp_{K_1} G_2}$. One can deduce a similar statement in the case of more connected components.}
\end{Remark}

The discussion so far yields a characterization of all graphs for which the cut ideal has ``small" regularity. It is reasonable to study the regularity of cut ideals in general. To the knowledge of the authors the only class of graphs for which an exact formula for the regularity of their cut ideals is known are trees. More precisely, if $T=(V,E)$ is a tree, then by \cite[Proposition~4.4]{NP},
\[
\reg_{S_T}I_T=|E|.
\]
Also, in some special cases, the regularity of cut ideals is bounded above by the number of edges plus one. For example the class of ring graphs is one of those cases (see, e.g., \cite[Corollary~6.5~and~Remark~6.6~(ii)]{NP}).

Applying Corollary~\ref{greatest generator}~(b), one can play a bit more with this invariant in some cases. In the following, we discuss such an example where a nontrivial lower bound for the regularity of the cut ideal of a special family of graphs is obtained.

\begin{Example}\label{unicyclic}
{\em Recall that a \emph{unicyclic graph} is a graph which has exactly one cycle as an induced subgraph. In particular, a unicyclic graph $G=(V,E)$ with $|V|=n$ is by definition a ring graph with $|E|=n-c+1$, where $c$ is the number of connected components of $G$.

Now, let $G$ be a connected unicyclic graph whose unique cycle is isomorphic to $C_m$ for some $m\geq 3$. By contracting $m-2$ edges of this cycle, we get a tree $T$ with $n-m+2$ vertices and $n-m+1$ edges. It follows from \cite[Proposition~4.4]{NP} that $\reg_{S_T} I_T=n-m+1$. Thus
\[
n-m+1\leq \reg_{S_G} I_G \leq n+1,
\]
where the lower bound follows from Corollary~\ref{greatest generator}~(b), since $\reg_{S_G} I_G\geq \reg_{S_T} I_T$. The upper bound is the one given in \cite[Corollary~6.5]{NP}.
}
\end{Example}

The following example shows that for a given natural number $r$, on can construct infinitely many graphs whose cut ideals have regularity at least $r$.

\begin{Example}\label{given reg}
{\em Suppose that $r,n\in \NN$ are given such that $n>r+1$. Also, let $T$ be any tree with $r+1$ vertices. We know by \cite[Proposition~4.4]{NP} that $\reg_{S_T} I_T=r$. Next let $H$ be any graph with $n-r-1$ vertices. Moreover, let $G$ be a graph with $n$ vertices such that $T$ is a combinatorial retract of it. Then
\[
\reg_{S_G} I_G \geq \reg_{S_T} I_T=r.
\]
Such a graph $G$ can be constructed as follows. One possibility is by edge contractions. It is also possible to choose $G$ as a clique-sum of $T$ and $H$ along $K_1$ or $K_2$. As a third option, $T$ can be a neighborhood-minor of $G$, i.e. $G$ can be obtained by joining at least one vertex of $H$ to at least one vertex  of any induced subgraph of $T$ which is isomorphic to a star graph.
}
\end{Example}

As we saw in Theorem~\ref{linear resolution} that for which graphs $G$ the minimal graded free resolution of $I_G$ is linear, it is also reasonable to ask for which cut ideals the resolution is linear up to a certain step.

Recall that a nonzero ideal $I$ in a polynomial ring $R=\KK[x_1,\ldots,x_n]$ is said to satisfy property $N_p$ , if it is generated in degree $2$, and its minimal graded free resolution is linear up to the $p$-th homological degree, i.e.
$\beta_{i,i+j}^R(I)=0$, for all $i\leq p$ and $j\neq 2$.

Note that, in Corollary~\ref{single degree}, cut ideals with property $N_0$ have been classified. It is natural to ask whether the property $N_p$ can be characterized in our setting. We concentrate here on the property $N_1$. Note that an ideal satisfying property $N_1$ is also called \emph{linearly presented}.

The next consequence of Corollary~\ref{greatest generator} is the following sufficient condition for satisfying the property $N_1$:

\begin{Corollary}\label{linearly presented}
Let $G=(V,E)$ be a graph with $|E|\geq 1$, and assume that $I_G$ satisfies property~$N_1$. Then $G$ is a
$(K_4,K_4\setminus e,C_4)$-combinatorial retract-free graph.
\end{Corollary}

\begin{proof}
Since $I_G$ satisfies property $N_1$, it follows by definition that $I_G$ is generated by quadrics, and hence $G$ is connected and $K_4$-minor-free by Proposition~\ref{disconnected} and Corollary~\ref{single degree}, respectively. Then, by Remark~\ref{combin retract K_n}, the graph $G$ is $K_4$-combinatorial retract-free.

We have that $\beta_{1,4}^{S_{K_4\setminus e}}(I_{K_4\setminus e})\neq 0$ and $\beta_{1,4}^{S_{C_4}}(I_{C_4})\neq 0$, because the cut algebras of $K_4\setminus e$ and $C_4$ are complete intersections by Theorem~\ref{complete intersection}, and their cut ideals which are not principal ideals by Proposition~\ref{projective dimension}, are generated by quadrics by Corollary~\ref{single degree}. Thus, if $G$ has a combinatorial retract isomorphic to $K_4\setminus e$ or $C_4$, then Corollary~\ref{greatest generator}~(a) implies that $\beta_{1,4}^{S_G}(I_G)\neq 0$, a contradiction to having property $N_1$. Hence, $G$ is $(K_4,K_4\setminus e,C_4)$-combinatorial retract-free.
\end{proof}

We would like to end this section by posing the following problem which is verified by computations for the graphs up to $5$ vertices; see Table~\ref{table:1}.

\begin{Problem}\label{linearly presented problem}
{\em What is the complete characterization of cut ideals of graphs which satisfy the property $N_1$? Is the sufficient condition given in Corollary~\ref{linearly presented} a necessary condition as well? }
\end{Problem}

\section{Examples and further remarks}\label{Examples}

In this section, we present some examples which have been essential in the literature and, in particular, in this paper. We also provide a table as a summary of some useful information about graphs with at most five vertices.

We start with two following examples concerning the complete graphs $K_4$ and $K_5$ as we mentioned in Section~\ref{properties and examples}.
Indeed, there exist some information about them in \cite[Table~1]{SS} which was determined by computations. In the following examples we study these two cases including rigorous proofs of ``well-known" facts. Recall the different gradings we introduced in Section~\ref{properties and examples} and will be used in the following.

\begin{Example}\label{K_4}
{\em We consider the complete graph $K_4$ on the vertex set $\{1,2,3,4\}$. By (\ref{height}), we have $\height I_{K_4}=1$ which implies that $I_{K_4}$ is a principal ideal, since it is a  prime ideal in the polynomial ring $S_G$. We note that it is easy to see that $u_1u_2u_3u_4=u_{\emptyset}u_{12}u_{13}u_{14}$, which implies that the binomial $f:=q_1q_2q_3q_4-q_{\emptyset}q_{12}q_{13}q_{14}$ belongs to $I_{K_4}$. We claim that $I_{K_4}=\langle f\rangle$.

Now, we prove the claim. Assume that $I_{K_4}=\langle g \rangle$ where $g$ is a pure binomial in $S_{K_4}$. It follows that $\deg(g)\leq 4$, because $\deg(f)=4$. On the other hand, since $K_4$ is connected, we have $\deg(g)\geq 2$ by Proposition~\ref{linear forms}. Since $f\in \langle g \rangle$, we have $f=hg$ for some homogeneous polynomial $h\in S_{K_4}$.

First suppose that $\deg(g)=2$, say $g=q_Aq_B-q_Cq_D$. Since $I_{K_4}$ is a prime ideal containing no linear forms, it follows that $q_Aq_B$ and $q_Cq_D$ have no common factors. Since $q_1q_2q_3q_4$, is in the support of $f$, it follows that either $q_Aq_B$ or $q_Cq_D$ divides $q_1q_2q_3q_4$. We may assume that $q_Aq_B|q_1q_2q_3q_4$, so that $|A|=|B|=1$. Without loss of generality, let $A=\{1\}$ and $B=\{2\}$. Then we have $\deg_s(q_1q_2)=6$, because $\deg_s(q_1)=\deg_s(q_2)=3$.

So, we deduce that $\deg_s(q_Cq_D)=6$ which implies that $|C|=|D|=1$. The latter follows, because if $|C|=0$ or $2$, then $\deg_s(q_C)=0$ or $4$, respectively, (similarly for $|D|$), and hence one can not get $6$ as $s$-degree of $q_Cq_D$. Hence, without loss of generality, let $C=\{3\}$ and $D=\{4\}$. But, this is a contradiction, since neither $q_1q_2$ nor $q_3q_4$ divides $q_{\emptyset}q_{12}q_{13}q_{14}$ which is also in the support of $f=hg$. Thus $\deg(g)\neq 2$.

Next assume that $\deg(g)=3$, say $g=q_Aq_Bq_C-q_Dq_Eq_F$. By the same argument as above, the monomials $q_Aq_Bq_C$ and $q_Dq_Eq_F$ do not have any common factors, and one of them, say $q_Aq_Bq_C$, divides $q_1q_2q_3q_4$. We may assume without loss of generality that $q_Aq_Bq_C=q_1q_2q_3$. Since $\deg_s(q_1q_2q_3)=9$, it also follows that $\deg_s(q_Dq_Eq_F)=9$ which is the case if and only if $|D|=|E|=|F|=1$, again by comparing the $s$-degrees of the variables. Therefore, $q_Aq_Bq_C$ and $q_Dq_Eq_F$ must have at least a variable as a common factor, because $K_4$ has only $4$ vertices, which is a contradiction. Hence
$\deg(g)\neq 3$.

Altogether we obtain that $\deg(g)=4$, which obviously implies that $g=f$ and $I_{K_4}=\langle f \rangle$, as we claimed. In particular, $S_{K_4}/I_{K_4}$ is a complete intersection, it is minimally resolved by the Koszul complex, and $\projdim_{S_{K_4}} S_{K_4}/I_{K_4}=1$.
}
\end{Example}

\begin{Example}\label{K_5}
{\em We consider the graph $K_5$ on the vertex set $V=\{1,2,3,4,5\}$. Here, we show that $\depth \KK[K_5]=1$, which then implies that $\KK[K_5]$ is not Cohen-Macaulay, according to the fact that
$\dim \KK[K_5]=11$ by (\ref{dim}). This, in addition, implies that $\KK[K_5]$ is not normal as well.

Let $T:=\KK[K_5]$. Note that $u_{\emptyset}$ is not a zero-divisor of $T$ which is an integral domain. In order to see that $\depth T=1$, it suffices to show that
the unique graded maximal ideal $\mathfrak{m}$ of the standard graded $\KK$-algebra $T$ is an associated prime ideal of $T/\langle u_{\emptyset}\rangle$, or equivalently, that there exists an element $f$ in $T$ such that $f\notin \langle u_{\emptyset}\rangle$ and $\mathfrak{m}=\langle u_{\emptyset}\rangle:_T f$. We set $f:=u_1u_2u_3u_4u_5$, and show that it has the desired properties. Observe that
\[
u_1u_2u_3u_{45}=u_{\emptyset}u_{12}u_{13}u_{23}
\]
and
\[
{u_1}^2u_2u_3u_4u_5={u_{\emptyset}}^2u_{12}u_{13}u_{14}u_{15}.
\]
Then, using the symmetry in $K_5$, we get the relations
\begin{equation}\label{relation1}
u_iu_ju_ku_{\ell p}=u_{\emptyset}u_{ij}u_{ik}u_{jk}
\end{equation}
and
\begin{equation}\label{relation2}
{u_i}^2u_ju_ku_{\ell}u_p={u_{\emptyset}}^2u_{ij}u_{ik}u_{i\ell}u_{ip},
\end{equation}
where $\{i,j,k,\ell,p\}=\{1,2,3,4,5\}$. Hence, it follows by the relations given in (\ref{relation1}) and (\ref{relation2}) that
$\mathfrak{m}\subseteq \langle u_{\emptyset}\rangle:_T f$, because the generators of $\mathfrak{m}$ are the monomials $u_{\emptyset}$, $u_i$'s and $u_{jk}$'s for
$i,j,k\in V$.

Next, we show that $f\notin \langle u_{\emptyset}\rangle$. Suppose on the contrary that $f\in \langle u_{\emptyset}\rangle$. Then, since $f$ is homogeneous with respect to multigrading, it follows that
\begin{equation}\label{zero devisor}
f=u_{\emptyset}u_{A}u_{B}u_{C}u_{D}
\end{equation}
for some subsets $A,B,C,D$ of $V$. Now, we use the $s$-degrees as in Example~\ref{K_4}. We have $\deg_s(u_{\emptyset})=0$,  $\deg_s(u_i)=4$, and $\deg_s(u_{jk})=6$ for distinct vertices $i,j,k\in V$. Thus $\deg_s(f)=20$, and hence according to (\ref{zero devisor}) we get $\deg_s(u_{A}u_{B}u_{C}u_{D})=20$. This is the case if and only if, up to a relabeling, $|A|=|B|=1$ and $|C|=|D|=2$. Without loss of generality, we assume that $A=\{1\}$ and $B=\{2\}$. Then, it follows from  (\ref{zero devisor}), that $u_3u_4u_5=u_{\emptyset}u_{C}u_{D}$, since $T$ is an integral domain. But, the latter equality can not occur, because $s^2_{34}s^2_{35}s^2_{45}$ divides the left-hand side. Indeed, this means that the edges $\{3,4\},\{3,5\},\{4,5\}$ belong to $\mathrm{Cut}(C)$ as well as to $\mathrm{Cut}(D)$, which is impossible. Therefore, we deduce that $f\notin \langle u_{\emptyset}\rangle$, as desired. This yields $\mathfrak{m}=\langle u_{\emptyset}\rangle:_T f$. In particular, by the Auslander-Buchsbaum formula, we have
\[
\projdim_{S_{K_5}} S_{K_5}/I_{K_5}=2^4-\depth S_{K_5}/I_{K_5}=15.
\]
}
\end{Example}

In Example~\ref{K_4}, it was shown that $I_{K_4}$ is a principal ideal, and in Example~\ref{K_5} some of the relations of $\KK[K_5]$ were presented. Indeed, (\ref{relation1}) and (\ref{relation2}) give us a combinatorial description of fifteen relations of degree~$4$, and five relations of degree~$6$, respectively. Note that the number of minimal generators of $I_{K_5}$ and $I_{K_6}$, respectively is known by \cite[Table~1]{SS}. For $n\geq 7$, not much more is known to the knowledge of the authors.
We would like to pose the following problem:

\begin{Problem}\label{problemK_n}
{\em Is there any nice combinatorial description of the generators of the cut ideal $I_{K_n}$ for $n\geq 5$? Moreover, it seems reasonable to investigate several algebraic properties of these ideals, as studied in the two cases here.}
\end{Problem}

In the  following  we consider some other graphs with small number of vertices which have been used throughout the paper.

\begin{Example}\label{another small}
{\em
\begin{enumerate}
\item [(a)] Let $G=P_3$ with $V(G)=\{1,2,3\}$, $E(G)=\{\{1,2\},\{2,3\}\}$. Note that by Proposition~\ref{projective dimension}, we have that $\projdim_{S_G}I_G=0$, and hence $I_G$ is a principal ideal. Moreover, it was observed in \cite[Example~2.3]{SS} that
\[
I_{G}=\langle q_{\emptyset}q_{2}-q_{1}q_{12}\rangle.
\]
\item [(b)] Let $G=C_4$ with $V(G)=\{1,2,3,4\}$ and
\[
E(G)=\{\{1,2\},\{2,3\},\{3,4\},\{1,4\}\}.
\]
As it was mentioned in \cite[Example~1.2]{SS}, a computation shows that
\[
I_{G}=\langle q_{\emptyset}q_{13}-q_{1}q_{3}, q_{\emptyset}q_{13}-q_{2}q_{4}, q_{\emptyset}q_{13}-q_{12}q_{14}\rangle
\]
which defines a complete intersection.
\item [(c)] Let $G=K_2\sharp_{K_1} K_3$ with $V(G)=\{1,2,3,4\}$ and
\[
E(G)=\{\{1,2\},\{2,3\},\{3,4\},\{2,4\}\}.
\]
Then a computation, e.g., with \emph{Macaulay2} (see \cite{GS}), shows that the Betti diagram of $S_G/I_G$ is the following which in particular shows that $I_G$ has a $2$-linear resolution:
{
\begin{verbatim}
                      0     1     2     3
               ----------------------------
               0:     1     -     -     -
               1:     -     6     8     3
               ----------------------------
           total:     1     6     8     3
\end{verbatim}
}

\item [(d)] Let $G=K_4\setminus e$. As it was observed in \cite[page~693]{SS},
\[
I_G=\langle q_{\emptyset}q_{14}-q_{1}q_{4}, q_2q_3-q_{12}q_{13} \rangle
\]
which defines a complete intersection.
\item [(e)] Let $G=K_5\setminus e$. As it was mentioned in \cite[Example~2.5]{SS}, a computation, e.g., with \emph{Macaulay2} (see \cite{GS}), shows that $I_G$ has $35$ minimal generators, and the Betti diagram of $S_G/I_G$ is the following:
{
\begin{verbatim}
             0     1     2     3     4     5     6
      ----------------------------------------------
      0:     1     -     -     -     -     -     -
      1:     -     4     -     -     -     -     -
      2:     -     -     6     -     -     -     -
      3:     -    31   128   200   128    31     -
      4:     -    -     -     -     6      -     -
      5:     -    -     -     -     -      4     -
      6:     -    -     -     -     -      -     1
      ----------------------------------------------
  total:     1   35   134   200   134     35     1
\end{verbatim}
}

\end{enumerate}
}
\end{Example}

Finally we present a table summarizing some information about all the graphs up to five vertices which do not have any isolated vertices and have nonzero cut ideals. The list of such graphs is taken from \cite[Appendix~1]{Ha}. Here we order the graphs in terms of the number of vertices, and for those which have the same number of vertices, the order is based on the number of edges. Those graphs for which there is no well-known notation, are denoted by $G_i$ for $i=1,\ldots,10$, and they are depicted in Figure~\ref{fig}.

\begin{figure}[h!]
\centering
{\begin{tikzpicture}[scale = 1]
\begin{scope}
\draw (0,0) node[punkt] {} -- (1,0) node[punkt] {};
\draw (0,0) node[punkt] {} -- (0,1) node[punkt] {};
\draw (0,0) node[punkt] {} -- (0.5,1.5) node[punkt] {};
\draw (1,0) node[punkt] {} -- (1,1) node[punkt] {};
\node [below] at (0.5,0) {$G_1$};
\end{scope}
\end{tikzpicture}\quad \quad \quad
\begin{tikzpicture}[scale = 1]
\begin{scope}
\draw (0,0) node[punkt] {} -- (1,0) node[punkt] {};
\draw (0,0) node[punkt] {} -- (0,1) node[punkt] {};
\draw (0,1) node[punkt] {} -- (1,0) node[punkt] {};
\draw (0,1) node[punkt] {} -- (0.5,1.5) node[punkt] {};
\draw (0.5,1.5) node[punkt] {} -- (1,1) node[punkt] {};
\node [below] at (0.5,0) {$G_2$};
\end{scope}
\end{tikzpicture}\quad \quad \quad
\begin{tikzpicture}[scale = 1]
\begin{scope}
\draw (0,0) node[punkt] {} -- (1,0) node[punkt] {};
\draw (0,0) node[punkt] {} -- (0,1) node[punkt] {};
\draw (0,1) node[punkt] {} -- (1,0) node[punkt] {};
\draw (0,1) node[punkt] {} -- (0.5,1.5) node[punkt] {};
\draw (1,0) node[punkt] {} -- (1,1) node[punkt] {};
\node [below] at (0.5,0) {$G_3$};
\end{scope}
\end{tikzpicture}\quad \quad \quad
\begin{tikzpicture}[scale = 1]
\begin{scope}
\draw (0,0) node[punkt] {} -- (1,0) node[punkt] {};
\draw (0,0) node[punkt] {} -- (0,1) node[punkt] {};
\draw (0,1) node[punkt] {} -- (1,0) node[punkt] {};
\draw (0,1) node[punkt] {} -- (0.5,1.5) node[punkt] {};
\draw (0,1) node[punkt] {} -- (1,1) node[punkt] {};
\node [below] at (0.5,0) {$G_4$};
\end{scope}
\end{tikzpicture}\quad \quad \quad
\begin{tikzpicture}[scale = 1]
\begin{scope}
\draw (0,0) node[punkt] {} -- (1,0) node[punkt] {};
\draw (0,0) node[punkt] {} -- (0,1) node[punkt] {};
\draw (0,0) node[punkt] {} -- (1,1) node[punkt] {};
\draw (0,1) node[punkt] {} -- (1,1) node[punkt] {};
\draw (1,0) node[punkt] {} -- (1,1) node[punkt] {};
\draw (1,1) node[punkt] {} -- (0.5,1.5) node[punkt] {};
\node [below] at (0.5,0) {$G_5$};
\end{scope}
\end{tikzpicture}\\
\begin{tikzpicture}[scale = 1]
\begin{scope}
\draw (0,0) node[punkt] {} -- (1,0) node[punkt] {};
\draw (0,0) node[punkt] {} -- (0,1) node[punkt] {};
\draw (0,0) node[punkt] {} -- (1,1) node[punkt] {};
\draw (0,1) node[punkt] {} -- (1,1) node[punkt] {};
\draw (1,0) node[punkt] {} -- (1,1) node[punkt] {};
\draw (0,1) node[punkt] {} -- (0.5,1.5) node[punkt] {};
\node [below] at (0.5,0) {$G_6$};
\end{scope}
\end{tikzpicture}\quad \quad \quad
\begin{tikzpicture}[scale = 1]
\begin{scope}
\draw (0,0) node[punkt] {} -- (1,0) node[punkt] {};
\draw (0,0) node[punkt] {} -- (0,1) node[punkt] {};
\draw (0,0) node[punkt] {} -- (1,1) node[punkt] {};
\draw (0,1) node[punkt] {} -- (1,1) node[punkt] {};
\draw (1,0) node[punkt] {} -- (1,1) node[punkt] {};
\draw (1,1) node[punkt] {} -- (0.5,1.5) node[punkt] {};
\draw (0,1) node[punkt] {} -- (0.5,1.5) node[punkt] {};
\node [below] at (0.5,0) {$G_7$};
\end{scope}
\end{tikzpicture}\quad \quad \quad
\begin{tikzpicture}[scale = 1]
\begin{scope}
\draw (0,0) node[punkt] {} -- (0,1) node[punkt] {};
\draw (0,0) node[punkt] {} -- (1,1) node[punkt] {};
\draw (0,1) node[punkt] {} -- (1,1) node[punkt] {};
\draw (1,0) node[punkt] {} -- (1,1) node[punkt] {};
\draw (1,1) node[punkt] {} -- (0.5,1.5) node[punkt] {};
\draw (0,1) node[punkt] {} -- (0.5,1.5) node[punkt] {};
\draw (0,1) node[punkt] {} -- (1,0) node[punkt] {};
\node [below] at (0.5,0) {$G_8$};
\end{scope}
\end{tikzpicture}\quad \quad \quad
\begin{tikzpicture}[scale = 1]
\begin{scope}
\draw (0,0) node[punkt] {} -- (0,1) node[punkt] {};
\draw (0,0) node[punkt] {} -- (1,0) node[punkt] {};
\draw (0,1) node[punkt] {} -- (1,1) node[punkt] {};
\draw (1,0) node[punkt] {} -- (1,1) node[punkt] {};
\draw (1,0) node[punkt] {} -- (0.5,1.5) node[punkt] {};
\draw (1,1) node[punkt] {} -- (0.5,1.5) node[punkt] {};
\draw (0,1) node[punkt] {} -- (0.5,1.5) node[punkt] {};
\node [below] at (0.5,0) {$G_9$};
\end{scope}
\end{tikzpicture}\quad \quad \quad
\begin{tikzpicture}[scale = 1]
\begin{scope}
\draw (0,0) node[punkt] {} -- (0,1) node[punkt] {};
\draw (0,0) node[punkt] {} -- (1,0) node[punkt] {};
\draw (0,1) node[punkt] {} -- (1,1) node[punkt] {};
\draw (1,0) node[punkt] {} -- (1,1) node[punkt] {};
\draw (1,0) node[punkt] {} -- (0.5,1.5) node[punkt] {};
\draw (1,1) node[punkt] {} -- (0.5,1.5) node[punkt] {};
\draw (0,1) node[punkt] {} -- (0.5,1.5) node[punkt] {};
\draw (0,0) node[punkt] {} -- (0.5,1.5) node[punkt] {};
\node [below] at (0.5,0) {$G_{10}$};
\end{scope}
\end{tikzpicture}
\caption{Some graphs from Table~\ref{table:1}}
\label{fig}
}
\end{figure}
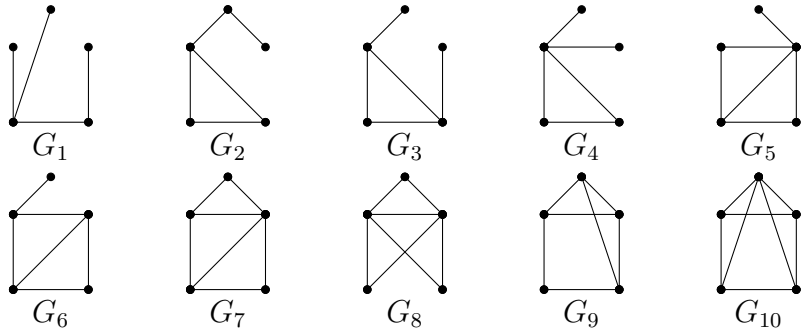

Beside the theoretical results, the data in Table~\ref{table:1} are based on computations by \emph{CoCoA} (see \cite{AB}) and \emph{Macaulay2} (see \cite{GS}). The terms used in the table for a graph $G=(V,E)$ are as follows:

\begin{itemize}
\item mindeg/maxdeg: the minimum/maximum degree in a minimal generating set of $I_G$;
\item projdim/reg: $\projdim_{S_G} I_G(=\projdim_{S_G} S_G/I_G-1)$/$\reg_{S_G} I_G$;
\item CM/Nor./Gor./C.I./ $N_1$: Cohen-Macaulay/Normal/Gorenstein/ Complete intersection/property $N_1$;
\item Y/N: Yes/No.
\end{itemize}
\begin{table}[ht]
\caption{Some properties and invariants of graphs (without isolated vertices) on $n\leq 5$ vertices. }
\centering
\begin{tabular}{|c|c|c|c|c|c|c|c|c|c|c|}
\hline
Graph & $|E|$ & mindeg & maxdeg &  projdim & reg & CM & Nor. & Gor. & C.I. & $N_1$  \\ [1ex]
\hline
\hline
$P_3$ & 2 & 2 & 2 & 0 & 2 & Y & Y & Y & Y & Y \\ \hline

$2K_2$ & 2 & 1 & 2 & 4 & 2 & Y & Y & Y & Y & N \\ \hline

$P_4$ & 3 & 2 & 2 & 3 & 3 & Y & Y & Y & N & Y \\ \hline

$K_{1,3}$ & 3 & 2 & 2 & 3 & 3 & Y & Y & Y & N & Y \\ \hline

$K_2\sharp_{K_1} K_3$ & 4 & 2 & 2 & 2 & 2 & Y & Y & N & N & Y \\ \hline

$C_4$ & 4 & 2 & 2 & 2 & 4 & Y & Y & Y & Y & N \\ \hline

$K_4\setminus e$ & 5 & 2 & 2 & 1 & 3 & Y & Y & Y & Y & N \\ \hline

$K_4$ & 6 & 4 & 4 & 0 & 4 & Y & Y & Y & Y & N \\ \hline

$K_2\sqcup P_3$ & 3 & 1 & 2 & 11 & 3 & Y & Y & Y & N & N \\ \hline

$K_2\sqcup K_3$ & 4 & 1 & 2 & 10 & 2 & Y & Y & N & N & N \\ \hline

$P_5$ & 4 & 2 & 2 & 10 & 4 & Y & Y & Y & N & Y \\ \hline

$K_{1,4}$ & 4 & 2 & 2 & 10 & 4 & Y & Y & Y & N & Y \\ \hline

$G_1$ & 4 & 2 & 2 & 10 & 4 & Y & Y & Y & N & Y \\ \hline

$G_2$ & 5 & 2 & 2 & 9 & 3 & Y & Y & N & N & Y \\ \hline

$G_3$ & 5 & 2 & 2 & 9 & 3 & Y & Y & N & N & Y \\ \hline

$G_4$ & 5 & 2 & 2 & 9 & 3 & Y & Y & N & N & Y \\ \hline

$K_2\sharp_{K_1} C_4$ & 5 & 2 & 2 & 9 & 5 & Y & Y & Y & N & N \\ \hline

$C_5$ & 5 & 2 & 2 & 9 & 4 & Y & Y & N & N & N \\ \hline

$K_3\sharp_{K_1} K_3$ & 6 & 2 & 2 & 8 & 4 & Y & Y & Y & N & Y \\ \hline

$G_5$ & 6 & 2 & 2 & 8 & 4 & Y & Y & N & N & N \\ \hline

$G_6$ & 6 & 2 & 2 & 8 & 4 & Y & Y & N & N & N \\ \hline

$K_3\sharp_{K_2} C_4$ & 6 & 2 & 2 & 8 & 5 & Y & Y & N & N & N \\ \hline

$C_4\sharp_{P_3} C_4$ & 6 & 2 & 2 & 8 & 6 & Y & Y & Y & N & N \\ \hline

$G_7$ & 7 & 2 & 2 & 7 & 5 & Y & Y & Y & N & N \\ \hline

$G_8$ & 7 & 2 & 2 & 7 & 5 & Y & Y & Y & N & N \\ \hline

$K_2\sharp_{K_1} K_4$ & 7 & 2 & 4 & 7 & 5 & Y & Y & N & N & N \\ \hline

$G_9$ & 7 & 2 & 2 & 7 & 5 & Y & Y & N & N & N \\ \hline

$K_3\sharp_{K_2} K_4$ & 8 & 2 & 4 & 6 & 5 & Y & Y & Y & N & N \\ \hline

$G_{10}$ & 8 & 2 & 2 & 6 & 6 & Y & Y & N & N & N \\ \hline

$K_5\setminus e$ & 9 & 2 & 4 & 5 & 7 & Y & Y & Y & N & N \\ \hline

$K_5$ & 10 & 4 & 6 & 14 & $\geq 6$ & N & N & N & N & N \\ \hline

\end{tabular}
\label{table:1}
\end{table}

\end{document}